\documentclass{amsart}

\usepackage[utf8]{inputenc} %unicode support

\usepackage{xargs} 
\usepackage{comment,amsmath,amssymb,amsthm,changepage,pifont,graphicx,tikz}
\usepackage[english]{babel}
\usetikzlibrary{patterns}
\usepackage[colorinlistoftodos,prependcaption,textsize=tiny]{todonotes}
\newcommandx{\unsure}[2][1=]{\todo[linecolor=red,backgroundcolor=red!25,bordercolor=red,#1]{#2}}
\newcommandx{\info}[2][1=]{\todo[linecolor=OliveGreen,backgroundcolor=OliveGreen!25,bordercolor=OliveGreen,#1]{#2}}
\allowdisplaybreaks

\newtheorem{theorem}{Theorem}[section]

\newtheorem{lemma}[theorem]{Lemma}
\newtheorem{corollary}[theorem]{Corollary}

\newtheorem{conjecture}[theorem]{Conjecture}

\theoremstyle{remark}
\newtheorem{remark}[theorem]{Remark}

\DeclareMathOperator{\lw}{lw}
\DeclareMathOperator{\conv}{conv}

\DeclareMathOperator{\ci}{Cliff}
\DeclareMathOperator{\Div}{div}
\DeclareMathOperator{\GL}{GL}

\DeclareMathOperator{\im}{im}
\newcommand{\RR}{\mathbb{R}}
\newcommand{\ZZ}{\mathbb{Z}}
\newcommand{\PP}{\mathbb{P}}

\newcommand{\floor}[1]{\left\lfloor #1 \right\rfloor}
\newcommand{\ceil}[1]{\left\lceil #1 \right\rceil}

\begin{document}

\title{Canonical syzygies of smooth curves on toric surfaces}

\author{Wouter Castryck, Filip Cools, Jeroen Demeyer, Alexander Lemmens}

\begin{abstract} 
In a first part of this paper, we prove constancy of
the canonical graded Betti table among the smooth curves in linear systems on Gorenstein weak Fano toric surfaces.
In a second part, we show that Green's canonical syzygy conjecture holds for all smooth curves of genus at most $32$ or Clifford index at most $6$ on arbitrary toric surfaces. Conversely we
use known results on Green's conjecture (due to Lelli-Chiesa) to obtain new facts about graded Betti tables of projectively embedded toric surfaces.\\

\end{abstract}

\maketitle

\section{Introduction}

Let $k$ be an algebraically closed field of characteristic zero, let $\Delta \subseteq \RR^2$ be a two-dimensional lattice polygon, and consider 
an irreducible Laurent polynomial 
\begin{equation} f = \sum_{(i,j) \in \Delta \cap \ZZ^2} c_{i,j} x^i y^j \in k[x^{\pm 1}, y^{\pm 1}] \label{eq_Laurentpoly} \end{equation}
that is supported on $\Delta$. Let $S_\Delta = k[ \, X_{i,j} \, | \, (i,j) \in \Delta \cap \ZZ^2 \,]$ be the polynomial ring obtained by associating a formal variable to each lattice point in $\Delta$. We
think of it as the homogeneous coordinate ring of the projective $(N_\Delta - 1)$-space, where $N_\Delta = | \Delta \cap \mathbb{Z}^2|$. Consider the map
$$ \varphi_\Delta : (k^*)^2 \hookrightarrow \PP^{N_\Delta - 1} : (\alpha,\beta) \mapsto (\alpha^i \beta^j)_{(i,j) \in \Delta \cap \ZZ^2}, $$
the Zariski closure of the image of which is a toric surface that we denote by $X_\Delta$.
%along with its Newton polygon
%\[ \Delta = \Delta(f) = \conv \{ \, (i,j) \in \ZZ^2 \, | \, c_{i,j} \neq 0 \, \} \subseteq \RR^2, \]
%which we assume to be two-dimensional.
Let $U_f$ be the curve in $(k^*)^2$ defined by $f=0$, and assume
that the closure $C_f$ of $\varphi_\Delta(U_f)$ inside $X_\Delta$ is a smooth hyperplane section, necessarily
cut out by
\[ \sum_{(i,j) \in \Delta \cap \ZZ^2} c_{i,j} X_{i,j} = 0. \]
This assumption is generically true, i.e., it holds for a dense open subset of the space of Laurent polynomials that are supported on $\Delta$. For instance, a well-known generically satisfied sufficient condition reads that
$f$ is $\Delta$-non-degenerate, in the sense of~\cite{batyrev}.

If $C_f$ is non-rational then every smooth and irreducible member in the complete linear system $|C_f|$ on $X_\Delta$ arises as $C_{f'}$ for some Laurent polynomial $f' \in k[x^{\pm 1}, y^{\pm 1}]$ that is supported on $\Delta$ and which is such that $C_{f'}$ is a smooth hyperplane section of $X_\Delta$. 
(If $C_f$ is rational then $|C_f|$ may contain torus-invariant prime divisors, which cannot be seen on $(k^\ast)^2$.) Furthermore, any such $C_{f'}$ is clearly lineary equivalent to $C_f$. In particular $|C_f|$ is parametrized by a dense open subset of the space $V_\Delta$ of Laurent polynomials that are supported on $\Delta$.

%If $f' \in k[x^{\pm 1}, y^{\pm 1}]$ is another Laurent polynomial which is supported on $\Delta$ such that $C_{f'}$ is a smooth hyperplane section of $X_\Delta$, then $C_f$ and $C_{f'}$ are clearly linearly equivalent when viewed as Weil divisors on $X_\Delta$. If $C_f$ is non-rational then all smooth and irreducible members in the linear system $|C_f|$ arise in this way. 

This holds generally: whenever one is given a complete linear system $|C|$ containing a smooth projective curve $C$ on a toric surface $X$,
say equipped with an embedding $\varphi: (k^*)^2 \hookrightarrow X$, then it arises in the above way.
Namely, let $P_C$ be the polygon associated with a torus-invariant divisor on $X$ that is linearly equivalent to $C$; 
see~\cite[\S4.3]{coxlittleschenck} for how this polygon is constructed.
Define $\Delta = \conv (P_C \cap \ZZ^2)$, where we note that if $C$ is Cartier then $P_C$ is a lattice polygon and one simply has $\Delta = P_C$.
If for $f$ one takes the generator of the ideal of $\varphi^{-1}(C)$ inside $k[x^{\pm 1}, y^{\pm 1}]$ that is supported on $\Delta$,
then the above assumptions are satisfied and one has $C_f \cong C$. For all other smooth projective curves $C' \in |C|$ one can similarly produce a Laurent polynomial $f' \in k[x^{\pm 1}, y^{\pm 1}]$ such that $C_{f'} \cong C'$
and such that $f'$ is also supported on $\Delta$. In particular, here too, the linear system $|C|$ is parametrized by a dense open subset of $V_\Delta$. We refer to~\cite[\S4]{linearpencils}
for more background on these claims.

\subsection*{Constancy of the gonality and the Clifford index}

Under our generic assumption that $C_f$ is a smooth hyperplane section, many of its
geometric invariants can be told explicitly from the combinatorics of $\Delta$. The starting result was proven by
Khovanskii~\cite{Khovanskii}, who obtained that the geometric genus $g(C_f)$ is given by $N_{\Delta^{(1)}} = | \Delta^{(1)} \cap \ZZ^2|$, where $\Delta^{(1)}$ denotes
the convex hull of the lattice points in the interior of $\Delta$ (in the cases where $\Delta^{(1)}$ is two-dimensional we will similarly write $\Delta^{(2)}$ to abbreviate $\Delta^{(1)(1)}$). To avoid low genus pathologies, from now on we will always assume that $| \Delta^{(1)} \cap \ZZ^2| \geq 4$. 
Then recent work of mainly Kawaguchi (a technical assumption was removed by the first two current authors) provides
a similar combinatorial interpretation for the Clifford index $\ci(C_f)$; see~\cite{CoppensMartens,ELMS} for some background on this invariant.

\begin{theorem}[see~\cite{linearpencils,kawaguchi}] \label{cliffindexthm} One has $\ci(C_f) = \lw(\Delta^{(1)})$
unless $\Delta^{(1)} \cong \Upsilon$, $\Delta^{(1)} \cong 2\Upsilon$ or $\Delta^{(1)} \cong (d-3)\Sigma$ for some $d\in\mathbb{Z}_{\geq 5}$, in which cases one has $\ci(C_f) = \lw(\Delta^{(1)}) - 1$.
\end{theorem}

\noindent Here $\lw$ denotes the lattice width \cite{CaCo} and $\Delta \cong \Delta'$ 
indicates that $\Delta'$ can be obtained from $\Delta$ using a linear transformation $\RR^2 \rightarrow \RR^2 : (x \ y) \mapsto (x \ y) A + b$, where $A \in \GL_2(\ZZ)$ and
$b \in \ZZ^2$. The polygons $\Upsilon$ and $\Sigma$ are respectively given by $\conv \{(-1,-1), (1,0), (0,1) \}$ and 
$\conv \{ (0,0), (1,0), (0,1) \}$, and the scalar multiples are in Minkowski's sense.
%\begin{center}
%\includegraphics[height=2.5cm]{UpsilonSigma.pdf}
 %\begin{tikzpicture}[scale=0.7]
   %\draw [thick] (-1,-1) -- (1,0) -- (0,1) -- (-1,-1);
   %\node at (-1.3,-1.25) {\footnotesize $(-1,-1)$};
   %\node at (1.65,0) {\footnotesize $(1,0)$};
   %\node at (0,1.3) {\footnotesize $(0,1)$};
 %\end{tikzpicture}
%\end{center}
%and will keep playing
%an exceptional role throughout this paper. 
As a corollary to Theorem~\ref{cliffindexthm}, we note that $C_f$ is non-hyperelliptic if and only if $\Delta^{(1)}$ is two-dimensional.

The proof of Theorem~\ref{cliffindexthm} entails similar interpretations for the gonality and the Clifford dimension. Finer data
that are known to be encoded in the combinatorics of $\Delta$ include the scrollar invariants~\cite{linearpencils} associated with a gonality pencil, which specialize
to the Maroni invariants in the trigonal case. Assuming that $\Delta$ satisfies
a mild condition, they also include the `scrollar ruling degrees' associated with a gonality pencil, which specialize to Schreyer's invariants $b_1, b_2$ in the case of tetragonal curves (where the mild condition is void); see~\cite{schreyersinv,intrinsicness}. 

An immediate consequence is that all these invariants depend on $\Delta$ only. This is an a priori non-trivial fact that can be rephrased as \emph{constancy} (of the Clifford index, the gonality, \dots) among
the smooth members in linear systems of curves on toric surfaces.
The existing literature contains other results of this type. For instance, work by Pareschi~\cite{pareschi} and Knutsen~\cite{knutsen_delpezzo} establishes constancy of the gonality and the Clifford index for curves in linear systems on Del Pezzo surfaces of degree at least two (recall that Del Pezzo surfaces are toric from degree six on). Recent work of Lelli-Chiesa extends this result to smooth rational surfaces $S$ whose anticanonical divisor $-K$ satisfies $h^0(S,-K) \geq 3$~\cite{lellichiesa}. Constancy of the gonality and the Clifford index may fail for linear systems on Del Pezzo surfaces of degree one; in fact this exception is also revisited by Lelli-Chiesa, who gives a natural sufficient condition for constancy in the cases where $H^0(S,-K) = 2$. 
Apart from rational surfaces, a theorem by Green and Lazarsfeld states that constancy of the Clifford index holds in linear systems on K3 surfaces~\cite{greenK3}. Here constancy of the gonality is not necessarily true, although it is known that there is only one counterexample, due to Donagi and Morrison; see~\cite{ciliberto_pareschi,knutsen}. 

\subsection*{Constancy results for the entire canonical Betti table}

In view of Theorem~\ref{cliffindexthm} and Green's canonical syzygy conjecture~\cite{greenkoszul} (see Conjecture~\ref{conj_Green} below),
it is natural to wonder whether similar constancy results hold for the entire graded Betti table 
\begin{equation} \label{curvebetti}
%\small
\begin{array}{r|cccccccc}
  & 0 & 1 & 2 & 3 & \dots & g-4 & g-3 & g - 2 \\
\hline
0 & 1 & 0 & 0 & 0 & \dots & 0 & 0 & 0 \\
1 & 0 & a_1 & a_2 & a_3 & \dots & a_{g-4} & a_{g-3} & 0 \\
2 & 0 & a_{g-3} & a_{g-4} & a_{g-5} & \dots & a_2 & a_1 & 0 \\ 
3 & 0 & 0 & 0 & 0 & \dots & 0 & 0 & \phantom{,}1, \\   
\end{array}
\end{equation}
of the canonical image of $C_f$ in $\PP^{g-1}$, where $g = g(C_f) = N_{\Delta^{(1)}}$. When writing down the above shape we use Serre duality in Koszul cohomology and 
we assume that $C_f$ is non-hyperelliptic or, equivalently, that $\Delta^{(1)}$ is two-dimensional, so that the canonical map $\kappa : C_f \rightarrow \PP^{g-1}$ is an embedding. We will
keep making this assumption throughout the rest of the article. An attractive feature of smooth curves in toric surfaces is that $\kappa$
is well understood.
Indeed, a refined version of Khovanskii's theorem provides us with a canonical divisor $K_\Delta$ on $C_f$ whose associated Riemann-Roch space
$H^0(C_f, K_\Delta)$ admits the basis $\{ \, x^i y^j \, | \, (i,j) \in \Delta^{(1)} \cap \ZZ^2 \, \}$. Thus for this choice of canonical divisor one has that 
\[ \kappa\circ\varphi_\Delta|_{U_f} = \varphi_{\Delta^{(1)}}|_{U_f}. \]
As a consequence the canonical model of $C_f$, which we denote by $C$, satisfies
\[ C \subseteq X_{\Delta^{(1)}} \subseteq \PP^{g-1}. \]
We therefore expect some interplay between the graded Betti table of $C$ and
that of $X_{\Delta^{(1)}}$, which is known to be of the form
\begin{equation} \label{toricbetti}
%\small
\begin{array}{r|ccccccc}
  & 0 & 1 & 2 & 3 & \dots & g-4 & g-3 \\
\hline
0 & 1 & 0 & 0 & 0 & \dots & 0 & 0 \\
1 & 0 & b_1 & b_2 & b_3 & \dots & b_{g-4} & b_{g-3} \\
2 & 0 & c_{g-3} & c_{g-4} & c_{g-5} & \dots & c_2 & \phantom{,}c_1, \\  
\end{array}
\end{equation}
by \cite[Lemma 1.2]{previouspaper}.

The main result of this article is the following constancy statement, whose proof is given in Section~\ref{section_constancyresults}.
We use $\partial \Delta^{(1)}$ to denote the boundary of $\Delta^{(1)}$. 
\begin{theorem} \label{maintheorem}
Let $\Delta \subseteq \RR^2$ be a lattice polygon 
such that $\Delta^{(1)}$ is two-dimensional and such that $g := | \Delta^{(1)} \cap \ZZ^2| \geq 4$. Let $f \in k[x^{\pm 1}, y^{\pm 1}]$ be supported on $\Delta$ and assume that $C_f$ is a smooth hyperplane section of the toric surface $X_\Delta$. Let $C$ be the canonical model of $C_f$, and for $\ell = 1, \ldots, g-3$ let $a_\ell$ (respectively, $b_\ell, c_\ell$) denote the graded Betti numbers of $C$ (resp., $X_{\Delta^{(1)}}$) as in~\eqref{curvebetti} (resp.,~\eqref{toricbetti}). If
\begin{itemize}
  \item the toric surface $X_{\Delta^{(1)}}$ associated with $\Delta^{(1)}$ is Gorenstein weak Fano, or
  \item $ | \partial \Delta^{(1)} \cap \ZZ^2| \geq g/2 +1$,
\end{itemize}
  then for all $\ell$ we have $a_\ell = b_\ell + c_\ell$. In particular, in these cases the graded Betti table of $C$ is independent of the coefficients of $f$.
\end{theorem}
\noindent Here, we recall that a normal surface is called Gorenstein weak Fano if its anticanonical divisor is a big and nef Cartier divisor. We refer to Section~\ref{section_gorensteinweakfano} below for a discussion of this notion in the toric case, including an easy rephrasing in combinatorial terms. As we will see, one is allowed to replace the condition that $X_{\Delta^{(1)}}$ is Gorenstein weak Fano by the condition that $X_\Delta$ is Gorenstein weak Fano, but then Theorem~\ref{maintheorem} becomes strictly weaker.

Two interesting classes of polygons which are covered by Theorem~\ref{maintheorem} are:
\begin{itemize}
  \item[\emph{(a)}] $\Delta \cong d\Sigma$ for some $d \geq 5$; this
  leads to the statement that the canonical graded Betti table
of a smooth degree $d$ curve in $\PP^2$ depends on $d$ only,
  \item[\emph{(b)}] $\Delta \cong [0,a] \times [0,b]$ for some pair of integers $a,b \geq 3$; this leads to the statement that
  smooth curves in $\PP^1 \times \PP^1$ of bidegree $(a,b)$ have a canonical graded Betti table which depends on $a$ and $b$ only.
\end{itemize} 
To our knowledge, both statements are new.
Slightly more generally, the theorem applies to the five Del Pezzo surfaces of degree at least six. This yields, for instance, that
\begin{itemize}
  \item[\emph{(c)}] the canonical graded Betti table of a degree $d \geq 5 + \lfloor \delta / 3 \rfloor$ curve in $\PP^2$ having $\delta \leq 3$ nodes in general position depends on $d$ and $\delta$ only.
\end{itemize}
Other examples of Gorenstein weak Fano toric surfaces include the weighted projective plane $\PP(1,1,2)$, which can be viewed as a quadric cone in $\PP^3$. Here Theorem~\ref{maintheorem} implies that 
\begin{itemize}
  \item[\emph{(d)}] the canonical graded Betti table of a smooth curve of weighted degree $2d$ in $\PP(1,1,2)$, for some integer $d \geq 3$, only depends on $d$.
\end{itemize}
Similarly:
\begin{itemize}
 \item[\emph{(e)}] the canonical graded Betti table of smooth curves of weighted degree $6d$ in $\PP(1,2,3)$, for some integer $d \geq 2$, only depends on $d$. 
\end{itemize} 
 The polygons corresponding to all these examples are depicted in Figure~\ref{gorweakfano_examples}.
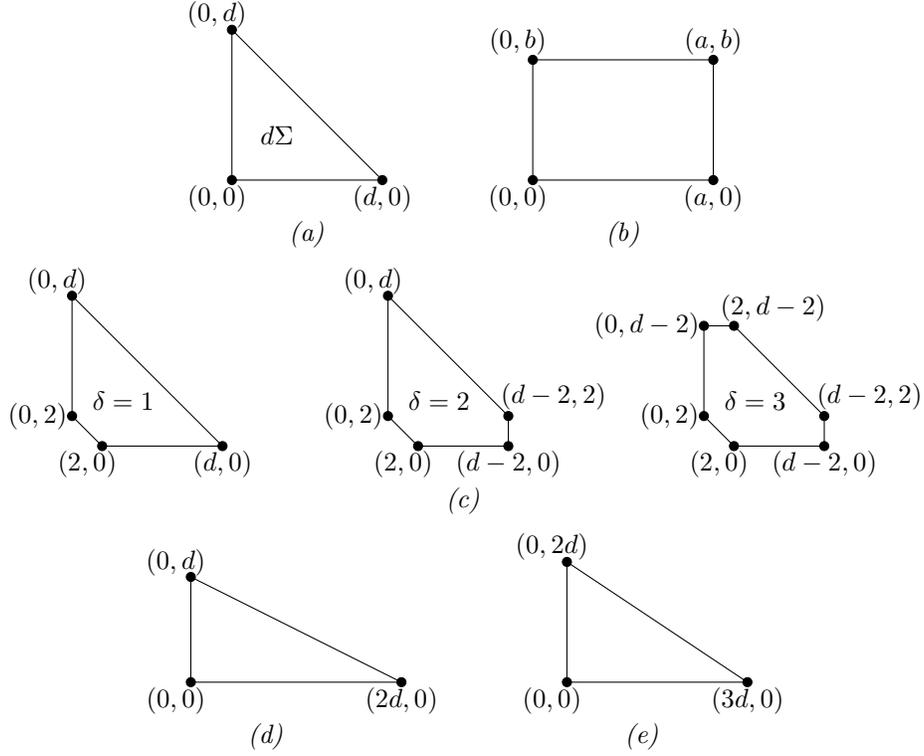
\begin{figure}[h]
\centering
\begin{tikzpicture}[scale=0.4]
  \draw (0,0)--(5,0) -- (0,5) -- (0,0);
  \node at (-0.5,-0.6) {$(0,0)$};
  \node at (5,-0.6) {$(d,0)$};
  \node at (-0.5,5.5) {$(0,d)$};
  \node at (1.5,1.5) {$d \Sigma$};
  \node at (2.5,-1.8) {\emph{(a)}};
  \draw[fill] (0,0) circle [radius=0.15];
  \draw[fill] (5,0) circle [radius=0.15];
  \draw[fill] (0,5) circle [radius=0.15];
    
    \draw (10,0)--(16,0) -- (16,4) -- (10,4) -- (10,0);
  \node at (9.5,-0.6) {$(0,0)$};
  \node at (16,-0.6) {$(a,0)$};
  \node at (9.5,4.6) {$(0,b)$};
  \node at (16,4.6) {$(a,b)$};
  \node at (13,-1.8) {\emph{(b)}};
  \draw[fill] (10,0) circle [radius=0.15];
  \draw[fill] (16,0) circle [radius=0.15];
  \draw[fill] (16,4) circle [radius=0.15];
  \draw[fill] (10,4) circle [radius=0.15];
\end{tikzpicture}

\begin{tikzpicture}[scale=0.4]
  \draw (1,0)--(5,0) -- (0,5) -- (0,1) -- (1,0);
  \node at (0.5,-0.6) {$(2,0)$};
  \node at (-1.15,1) {$(0,2)$};
  \node at (5,-0.6) {$(d,0)$};
  \node at (-0.5,5.5) {$(0,d)$};
  \node at (1.7,1.5) {$\delta = 1$};
  \draw[fill] (1,0) circle [radius=0.15];
  \draw[fill] (5,0) circle [radius=0.15];
  \draw[fill] (0,5) circle [radius=0.15];
  \draw[fill] (0,1) circle [radius=0.15];
  
  \draw (11.5,0)--(14.5,0) -- (14.5,1) -- (10.5,5) -- (10.5,1) -- (11.5,0);
  \node at (11,-0.6) {$(2,0)$};
  \node at (9.35,1) {$(0,2)$};
  \node at (14.5,-0.6) {$(d-2,0)$};
  \node at (16,1.65) {$(d-2,2)$};
  \node at (10,5.5) {$(0,d)$};
  \node at (12.2,1.5) {$\delta = 2$};
  \node at (13,-1.8) {\emph{(c)}};
  \draw[fill] (11.5,0) circle [radius=0.15];
  \draw[fill] (14.5,0) circle [radius=0.15];
  \draw[fill] (14.5,1) circle [radius=0.15];
  \draw[fill] (10.5,5) circle [radius=0.15];
  \draw[fill] (10.5,1) circle [radius=0.15];
  
  \draw (22,0)--(25,0) -- (25,1) -- (22,4) -- (21,4) -- (21,1) -- (22,0);
  \node at (21.5,-0.6) {$(2,0)$};
  \node at (19.85,1) {$(0,2)$};
  \node at (25,-0.6) {$(d-2,0)$};
  \node at (26.5,1.65) {$(d-2,2)$};
  \node at (19.1,4) {$(0,d-2)$};
  \node at (23.3,4.6) {$(2,d-2)$};
  \node at (22.7,1.5) {$\delta = 3$};
  \draw[fill] (22,0) circle [radius=0.15];
  \draw[fill] (25,0) circle [radius=0.15];
  \draw[fill] (25,1) circle [radius=0.15];
  \draw[fill] (22,4) circle [radius=0.15];
  \draw[fill] (21,4) circle [radius=0.15];
  \draw[fill] (21,1) circle [radius=0.15];
\end{tikzpicture}

\begin{tikzpicture}[scale=0.4]
  \draw (0,0)--(7,0) -- (0,3.5) -- (0,0);
  \node at (-0.5,-0.6) {$(0,0)$};
  \node at (7,-0.6) {$(2d,0)$};
  \node at (-0.5,4) {$(0,d)$};
  \draw[fill] (0,0) circle [radius=0.15];
  \draw[fill] (7,0) circle [radius=0.15];
  \draw[fill] (0,3.5) circle [radius=0.15];
  \node at (2.5,-1.8) {\emph{(d)}};
  
  \draw (12.5,0)--(18.5,0) -- (12.5,4) -- (12.5,0);
  \draw[fill] (12.5,0) circle [radius=0.15];
  \draw[fill] (18.5,0) circle [radius=0.15];
  \draw[fill] (12.5,4) circle [radius=0.15];
  \node at (12,-0.6) {$(0,0)$};
  \node at (18.5,-0.6) {$(3d,0)$};
  \node at (12,4.5) {$(0,2d)$};
  \node at (15,-1.8) {\emph{(e)}};
\end{tikzpicture}

\caption{Polygons to which Theorem~\ref{maintheorem} should be applied in order to cover examples \emph{(a-e)}}
\label{gorweakfano_examples}
\end{figure}

The class of polygons $\Delta$ for which $ | \partial \Delta^{(1)} \cap \ZZ^2| \geq g/2 + 1$, on the other hand, covers all cases where $\lw(\Delta^{(1)}) \leq 2$ by~\cite[Lem.\,9]{schreyersinv}. Such polygons correspond to
trigonal and certain tetragonal curves, where constancy was known to hold before~\cite{schreyersinv,schreyer}.

We actually believe that the sum formula $a_\ell = b_\ell + c_\ell$ is true for a considerably larger class of polygons
than the ones covered by the above theorem. Of course, even when the formula fails, it might still be true that the graded Betti table of $C$ does not depend on $f$, i.e., the
defect might depend on $\Delta$ and $\ell$ only. Examples of such behaviour are given in Section~\ref{section_constancyresults}. We leave it as an open question
whether or not this is true in general.

\subsection*{New cases of Green's conjecture}
In Section~\ref{section_greenconnections} we study connections between Green's canonical syzygy conjecture and a conjecture on graded Betti tables of toric surfaces that we have stated in a previous article~\cite{previouspaper}:
\begin{conjecture}[Green] \label{conj_Green} Let $C / k$ be a smooth
projective non-hyperelliptic curve of genus $g \geq 4$. Denote the graded Betti table of its canonical model in $\PP^{g-1}$ as in~\eqref{curvebetti}. Then 
$\min\{\ell\,|\,a_{g-\ell}\neq 0\} = \ci(C)+2$.
\end{conjecture}
\noindent (The requirement that $C$ is non-hyperelliptic is only included for compatibility with the above discussion; the full version of Green's conjecture naturally covers hyperelliptic curves too, where it amounts to a well-known fact; see e.g.~\cite{beauville,schreyer}, which also contain more background on the role of the base field $k$.)

\begin{conjecture} \label{conj_CCDL} Let $\Delta \subseteq \RR^2$ be a lattice polygon whose interior polygon $\Delta^{(1)}$ is two-dimensional
and contains $g \geq 4$ lattice points. Assume that $\Delta^{(1)}\not\cong \Upsilon$. If we denote the graded Betti table of $X_{\Delta^{(1)}} \subseteq \mathbb{P}^{g-1}$ as in \eqref{toricbetti}, then we have that 
$$\min\{\ell\,|\,b_{g-\ell}\neq 0\} = \begin{cases} \lw(\Delta^{(1)})+1 &\mbox{if } \Delta^{(1)}\cong (d-3)\Sigma \text{ for some } d\geq 5, \\ \lw(\Delta^{(1)})+1 &\mbox{if } \Delta^{(1)}\cong 2\Upsilon, \\ \lw(\Delta^{(1)})+2 &\mbox{in all other cases.}\end{cases}$$
\end{conjecture} 
\noindent In both conjectures the right hand side of the predicted equality is known to be an upper bound. We note that the actual version of Conjecture~\ref{conj_CCDL}, as it is formulated in~\cite{previouspaper}, covers arbitrary projectively embedded toric surfaces, i.e., not necessarily of
the form $X_{\Delta^{(1)}}$. But since this more general version is of no use to the current article, we omit it. 

Concretely, it is not hard to establish the following connection between the two conjectures (a proof will be given in Section~\ref{section_greenconnections}):
\begin{lemma} \label{lemma2conj} If
Conjecture \ref{conj_Green} holds for smooth irreducible hyperplane sections of $X_\Delta$ then Conjecture \ref{conj_CCDL} correctly predicts the length of the linear strand of
the graded Betti table of $X_{\Delta^{(1)}}$. If $|\partial \Delta^{(1)}\cap \mathbb{Z}^2|\geq \lw(\Delta^{(1)})+2$ then also the converse implication holds. 
\end{lemma}
\noindent We use this to settle new cases of both Conjecture~\ref{conj_Green} and Conjecture~\ref{conj_CCDL}.

Namely, in~\cite{previouspaper} we proved that Conjecture~\ref{conj_CCDL} is true as soon as $g = N_{\Delta^{(1)}} \leq 32$ or $\lw(\Delta^{(1)}) \leq 6$, a claim which relies in part on an explicit computational verification using the data from~\cite{movingout}. As we will see, the condition $|\partial \Delta^{(1)}\cap \mathbb{Z}^2|\geq \lw(\Delta^{(1)})+2$ is always satisfied in these ranges, except in the genus $4$ case where $\Delta^{(1)} \cong \Upsilon$, which is of no concern. Through Lemma~\ref{lemma2conj} this yields: 

\begin{theorem} \label{thm_greenlowgenuslowcliff}
Green's conjecture holds for all smooth curves $C / k$ on toric surfaces of genus $g \leq 32$ or Clifford index $\ci(C) \leq 6$.
\end{theorem}

Conversely, Lelli-Chiesa's aforementioned work~\cite{lellichiesa} settles Green's conjecture for smooth curves on smooth rational surfaces whose anticanonical divisor has enough sections; see Theorem~\ref{lellichiesatheorem} for a precise formulation of the latter condition. We will show that in the case of smooth toric surfaces $X$ with canonical divisor $K$, the condition is equivalent to $h^0(X, -K) \geq 2$. A short reasoning then allows us to conclude that Green's conjecture holds for all smooth hyperplane sections of $X_\Delta$, for any lattice polygon $\Delta$ which satisfies $h^0(X_{\Delta^{(1)}}, -K_{\Delta^{(1)}}) \geq 2$. As before, it is assumed that $\Delta^{(1)}$ is two-dimensional and that $|\Delta^{(1)} \cap \ZZ^2| \geq 4$, and $K_{\Delta^{(1)}}$ denotes a canonical divisor on $X_{\Delta^{(1)}}$. Then Lemma~\ref{lemma2conj} implies:

\begin{theorem} \label{thm_newcasestoricconjecture}
 Conjecture~\ref{conj_CCDL} holds for all lattice polygons $\Delta$ such that $\Delta^{(1)}$ is two-dimensional, contains at least $4$ lattice points, and satisfies $h^0(X_{\Delta^{(1)}}, -K_{\Delta^{(1)}}) \geq 2$.
\end{theorem}

\noindent The condition $h^0(X_{\Delta^{(1)}}, -K_{\Delta^{(1)}}) \geq 2$ has an easy and well-known combinatorial interpretation, which is recalled in Section~\ref{section_greenconnections} (it appears as a proof ingredient there).

%In fact, so far we have found one counterexample only, in genus $g = 12$.
% Moreover, if the formula fails then we conjecture that the defect depends on $\Delta$ and $\ell$ only, i.e.\ that it remains true in general 
% that the graded Betti table is independent of the specific choice of $f$.\\

\section{An exact sequence involving six terms} \label{sectionsixterms}

Let $\Delta$ be a lattice polygon with a two-dimensional interior lattice polygon $\Delta^{(1)}$. Let 
$f \in k[x^{\pm 1}, y^{\pm 1} ]$ be an irreducible Laurent polynomial as in \eqref{eq_Laurentpoly} and assume
that $C_f$ is a smooth hyperplane section of $X_\Delta$. 
Let $\rho : X \rightarrow X_\Delta$ be the \emph{minimal} toric resolution of singularities, i.e., $X$ is
the toric surface associated with the smooth subdivision of the inner normal fan to $\Delta$ in which no more
new rays are introduced than needed (remember that a subdivision is smooth if and only if the corresponding toric surface is smooth, which holds if and only if the primitive generators of each pair of adjacent rays form a basis of $\ZZ^2$ as a $\ZZ$-module). It can be obtained using Hirzebruch-Jung continued fractions
as described in \cite[\S10.2]{coxlittleschenck}. 
%In fact the condition of minimality is not strictly needed
%for what follows, but it has some technical advantages, mainly in Lemma~\ref{fujitalemma} below. 
Let $K$ be the canonical divisor on $X$ obtained
by taking minus the sum of all torus-invariant prime divisors \cite[Thm.\,8.2.3]{coxlittleschenck}.

Because $C_f$ does not meet the singular locus of $X_\Delta$, it pulls
back to an isomorphic curve $C'$ on $X$. 
Define $D_f = C' - \Div(f)$, where $f$ is viewed as a function on $X$ by pushing it forward along $\varphi_\Delta$ and then pulling it back along $\rho$. This is a torus-invariant
divisor that is linearly equivalent to $C'$. 

\begin{lemma} \label{fujitalemma}
Letting $\Delta$, $f \in k[x^{\pm 1}, y^{\pm 1} ]$, $D_f$ and $K$ be as above, one has:
\begin{itemize}
  \item the divisor $D_f$ is base-point free, and the polygon $P_{D_f}$ associated with $D_f$ is $\Delta$,
  \item its adjoint divisor $L := D_f + K$ is also base-point free, and the polygon $P_L$
  associated with $L$ is $\Delta^{(1)}$.
\end{itemize}
\end{lemma}

\noindent The second statement might be of interest to people studying 
Fujita type results; see \cite{laterveer,mustata}.
Here the minimality of our resolution $X \rightarrow X_\Delta$ is important, as the reader can tell from the proof below.
Also recall that for divisors on a smooth toric surface, the notions of base-point free and nef 
are synonymous~\cite[Thms.\,6.1.7 and 6.3.12]{coxlittleschenck}.

\begin{proof}
Let $\Sigma_\Delta$ be the fan of $X_\Delta$ (i.e., the inner normal fan to $\Delta$) and let $\Sigma$ be the fan of $X$. Denote by $U(\Sigma)$ the set of primitive generators of the rays of $\Sigma$, and let $U(\Sigma_\Delta)\subseteq U(\Sigma)$ be 
the subset of vectors that correspond to rays of $\Sigma_\Delta$. Since the divisor $D_f$ is torus-invariant, it is of the form $\sum_{v\in U(\Sigma)} a_v D_v$, where $D_v\subseteq X$ is the prime divisor corresponding to the ray generated by $v$. Let $H(v,a_v)$ be the half-plane of points $x\in\mathbb{R}^2$ satisfying $\langle x,v\rangle\geq -a_v$ and let $L(v,a_v)$ be the line defined by $\langle x,v\rangle=-a_v$. 
As explained in \cite[\S4]{linearpencils}, we have that 
\begin{equation} \label{formulaforPDf}
P_{D_f}=\bigcap_{v\in U(\Sigma)}\, H(v,a_v)=\bigcap_{v\in U(\Sigma_\Delta)}\, H(v,a_v)=\Delta.
\end{equation}
Moreover, if $u\in U(\Sigma)\setminus U(\Sigma_\Delta)$ corresponds to a ray that lies in between two consecutive rays of $\Sigma_\Delta$ with primitive generators $v,w\in U(\Sigma_\Delta)$, then $L(u,a_u)$ passes through the vertex $L(v,a_v)\cap L(w,a_w)$ of $\Delta$. In other words, if $v,w\in U(\Sigma)$ correspond to consecutive rays of $\Sigma$, then $L(v,a_v)\cap L(w,a_w)\in \Delta$. By \cite[Prop.\,6.1.1]{coxlittleschenck}, this just means that $D_f$ is base-point free (which also follows directly from the fact that $C'$ is the pull-back of the base-point free divisor $C_f$ on $X_\Delta$, but we will reuse this combinatorial criterion below).

Since $K=-\sum_{v\in U(\Sigma)} D_v$, we have that $L=\sum_{v\in U(\Sigma)} (a_v-1) D_v$. It follows that the polygon associated with $L$ is 
\begin{equation} \label{formulaforPL} 
 P_L=\bigcap_{v\in U(\Sigma)}\, H(v,a_v-1).
\end{equation}
To prove that $L$ is base-point free, again
by \cite[Prop.\,6.1.1]{coxlittleschenck} it suffices to show that for all $v,w\in U(\Sigma)$ that correspond to adjacent rays, the lattice point $m_1=L(v,a_v-1)\cap L(w,a_w-1)$ belongs to $P_L$.
Because $X$ is smooth, the vectors $v,w$ form a basis of $\mathbb{Z}^2$, hence using a unimodular transformation if needed we may assume that $v=(1,0)$ and $w=(0,1)$. Then the point $m_1$ becomes $(-a_v+1,-a_w+1)$. From the base-point-freeness of $D_f$ we know that $$m=(-a_v,-a_w)\in \Delta\subseteq H(v,a_v)\cap H(w,a_w).$$ 

Now consider $v',w'\in U(\Sigma_\Delta)$ such that $m\in L(v',a_{v'})\cap L(w',a_{w'})$, so $v'$ and $w'$ are the primitive normal vectors of the edges of $\Delta$ that are adjacent to the vertex $m$. We can assume that $L(v',a_{v'})$ is steeper than $L(w',a_{v'})$, and note that it could happen that $v'=v$ and/or $w'=w$. In order to prove that $m_1 \in P_L$, it suffices to show that $L(v',a_{v'})$ passes strictly above $m_1$ and that $L(w',a_{w'})$ passes strictly below $m_1$. We only prove the statement for $v'$; the one for $w'$ follows by symmetry.  
\begin{figure}[h]
\centering
\begin{tikzpicture}[scale=0.4]
  \draw[dashed] (-1,0)--(5,0);
  \draw[dashed] (0,-1)--(0,5);
  \draw[dashed] (-1,-0.4)--(5,2);
  \draw[dashed] (-0.4,-1)--(1.75,4.375);
  \draw[fill] (0,0) circle [radius=0.15];
  \draw[fill] (1,1) circle [radius=0.15];
  \draw[thick] (0,0) -- (2.5,1) -- (4.5,2.4) -- (4,4) -- (3,4) -- (1,2.5) -- (0,0);
  \node at (4.7,4) {$\Delta$};
  \node at (1.4,1.4) {$m_1$};
  \node at (-0.5,0.4) {$m$};
  \node at (4,-0.6) {$L(w,a_w)$};
  \node at (5.6,1) {$L(w',a_{w'})$};
  \node at (-1.6,4.8) {$L(v,a_v)$};
  \node at (2.2,4.8) {$L(v',a_{v'})$};  
\end{tikzpicture}
\qquad \qquad 
\begin{tikzpicture}[scale=0.4]
  \draw[->] (0,0) -- (3.75,-1.5);  
  \draw[->] (0,0) -- (-1.5,3.75);
  \draw[->] (0,0) -- (2,0);
  \draw[->] (0,0) -- (0,2);
  \node at (4,-1.75) {$v'$};
  \node at (4,0) {$v=(1,0)$};
  \node at (1.2,2.5) {$w=(0,1)$};
  \node at (-2,3.75) {$w'$};
\end{tikzpicture}
\caption{Rays of $X$ and $X_\Delta$ that are adjacent to $m$}
\label{PandQ}
\end{figure}
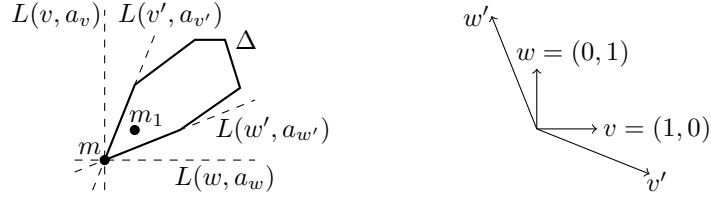

Let $v_0=v,v_1,\ldots,v_n=v'$ be the vectors in $U(\Sigma)$ from $v$ up to $v'$ going clockwise.  
We claim that all $v_i$ satisfy $x_i > -y_i$, where $v_i = (x_i,y_i)$. 
For $i=n$, this claim tells us that $L(v',a_{v'})$ passes strictly above $m_1$. 
Suppose our claim is false and let $i$ be minimal such that $x_i \leq -y_i$. Note that $i>0$. 
It is impossible that $x_i = -y_i$, because in that case $w=(0,1)$ and $v_i=(1,-1)$ would
be a basis of $\mathbb{Z}^2$, so would be able to delete the rays corresponding to 
$v_j\in U(\Sigma)$ with $j < i$, while the associated toric surface would still be a resolution of singularities of $X_{\Delta}$, contradicting the minimality assumption. So
$x_i<-y_i$. Also $x_{i-1} > -y_{i-1}$, by the minimality of $i$. Now $v_{i-1}$ and $v_i$ must form
a basis of $\mathbb{Z}^2$ and hence the determinant $x_i y_{i-1} - x_{i-1} y_i$ of the matrix formed
by $v_i$ and $v_{i-1}$ is $\pm 1$. But $x_{i-1}(-y_i) > x_{i-1}x_i > (-y_{i-1})x_i$, and since we
have two strict inequalities, the difference is at least $2$. This contradicts that
the determinant is $\pm 1$, proving our claim.

It remains to show that $P_L = \Delta^{(1)}$. Because $L$ is base-point free, again from the criterion \cite[Prop.\,6.1.1]{coxlittleschenck} we see that $P_L$ is a lattice polygon. From~\eqref{formulaforPDf} and~\eqref{formulaforPL} one sees that $P_L$ is contained in the topological interior of $\Delta$. On the other hand $\Delta^{(1)} \subseteq P_L$ because every interior lattice point lies at integral distance at least $1$ from the boundary. The desired conclusion follows. 
\end{proof} 

The above lemma is valuable in investigating how the graded Betti table~\eqref{curvebetti} of the canonical model $C$ of $C_f$ relates to the graded Betti table~\eqref{toricbetti} of $X_{\Delta^{(1)}}$. 
We assume that the reader is familiar with how the entries $a_\ell, b_\ell, c_\ell$ for $\ell = 1, \ldots, g-3$ arise as dimensions of Koszul cohomology spaces.
We refer to~\cite{nagelaprodu} for more background, and to \cite[\S2]{previouspaper} and~\cite{heringphd}
for a discussion that is specific to toric surfaces. For what follows, it is convenient to define $a_0 = b_0 = c_0 = a_{g-2} = b_{g-2} = c_{g-2} = 0$. 

Our starting point is the standard exact sequence
$0 \rightarrow \mathcal{O}_X(-C') \rightarrow \mathcal{O}_X \rightarrow \mathcal{O}_{C'} \rightarrow 0$
of sheaves of $\mathcal{O}_X$-modules.
It can be rewritten as
\[ 0 \rightarrow \mathcal{O}_X(-D_f) \stackrel{\mu_f}{\longrightarrow} \mathcal{O}_X \longrightarrow \mathcal{O}_{C'} \rightarrow 0, \]
where $\mu_f$ denotes multiplication by the function $f$. By the adjunction formula $K_{C'} := L|_{C'}$
is a canonical divisor on $C'$. Tensoring the above exact sequence with $\mathcal{O}_X(qL)$ then gives exact sequences
\[
 0 \rightarrow \mathcal{O}_X((q-1)L + K) \stackrel{\mu_f}{\longrightarrow} \mathcal{O}_X(qL) \longrightarrow \mathcal{O}_{C'}(qK_{C'}) \rightarrow 0 
\]
for all $q \geq 0$.
We claim that
 $H^1(X, (q-1)L + K) = 0$, which
by Serre duality~\cite[Thm.\,9.2.10]{coxlittleschenck} is equivalent to $H^1(X,(1-q)L) = 0$. Indeed for $q = 0$ and $q = 1$
this is true by Demazure vanishing~\cite[Thm.\,9.2.3]{coxlittleschenck}, while for $q \geq 2$ it follows
from Batyrev-Borisov vanishing~\cite[Thm.\,9.2.7(a)]{coxlittleschenck}. In both cases we used that $L$ is base-point free, while in the latter case we also used that $P_L = \Delta^{(1)}$
is two-dimensional. Thus by taking cohomology we obtain a short exact sequence
\[  0 \rightarrow \bigoplus_{q \geq 0} H^0(X, (q-1)L + K) \stackrel{\mu_f}{\longrightarrow} \bigoplus_{q \geq 0} H^0(X,qL) \longrightarrow \bigoplus_{q \geq 0} H^0(C',qK_{C'}) \rightarrow 0 \]
of $k$-vector spaces. In a natural way, this can be viewed as an exact sequence of graded modules over $S_{\Delta^{(1)}}=S^*V_{\Delta^{(1)}}$, where
$V_{\Delta^{(1)}} = H^0(X,L)$ and $S^*$ denotes the symmetric algebra. This claim relies on the fact that $H^0(X,L) \cong H^0(C',K_{C'})$, which is true because $H^0(X,K) = 0$ in the case of toric surfaces. The notation $V_{\Delta^{(1)}}$ is taken from~\cite[\S2]{previouspaper} and emphasizes that $H(X,L)$ can be viewed as the subspace of $k[x^{\pm 1}, y^{\pm 1}]$ consisting of those Laurent polynomials that are supported on $P_L = \Delta^{(1)}$. By \cite[Cor.\,(1.d.4)]{greenkoszul} or \cite[Lem.\,1.25]{nagelaprodu}
we find a long exact sequence in Koszul cohomology:
\[ \dots \rightarrow K_{p,q-1}(X;K,L) \stackrel{\mu_f}{\longrightarrow} K_{p,q}(X,L) \longrightarrow K_{p,q}(C',K_{C'}) \qquad \qquad \qquad \qquad \qquad \]
\[ \qquad \qquad \longrightarrow K_{p-1,q}(X;K,L) \stackrel{\mu_f}{\longrightarrow} K_{p-1,q+1}(X,L) 
\longrightarrow K_{p-1,q+1}(C',K_{C'}) \rightarrow \dots \]
%where the attribute `$;K$' denotes Koszul cohomology twisted by $K$, as described in~\cite[\S1.4]{nagelaprodu}.
Now note that the image of
$X \stackrel{| L |}{\longrightarrow} \mathbb{P}^{g - 1}$, where $g=h^0(X,L)=|\Delta^{(1)}\cap\mathbb{Z}^2|$,
is nothing else but $X_{\Delta^{(1)}}$. Thus
\begin{align*}
 b_\ell & = \dim K_{\ell,1}(X,L) = \dim K_{g-3-\ell,2}(X;K,L), \\ 
 c_\ell & = \dim K_{g - 2 - \ell,2}(X,L) = \dim K_{\ell-1,1}(X;K,L)
\end{align*}
for $\ell = 0, 1, \ldots, g - 2$,
where the last equalities again follow from Serre duality, as explained in more detail in~\cite[\S2.1]{previouspaper}.
%in Koszul cohomology~\cite[Thm.\,2.25]{nagelaprodu}
%along with Demazure vanishing and the fact that $L$ is base-point free. 
Combining these formulas with $a_\ell = \dim K_{\ell,1}(C',K_{C'})$ we find
for each $\ell=0, 1, \ldots, g - 2$ our desired exact sequence, of the form
\begin{equation} \label{sixterms} 
  0 \rightarrow b_\ell \rightarrow a_\ell \rightarrow c_\ell \stackrel{\mu_f}{\longrightarrow} c_{g - 1 - \ell} \rightarrow a_{g - 1 - \ell} \rightarrow b_{g - 1 - \ell} \rightarrow 0
\end{equation}
where we abusingly write the dimensions, rather than the cohomology spaces themselves.

\begin{remark} It follows that $$b_\ell + c_\ell - c_{g - 1 - \ell} - b_{g - 1 - \ell} = a_\ell - a_{g - 1 - \ell}.$$ 
The right hand side is known to be equal to 
\[ \binom{g - 1}{\ell - 1} \frac{(g - 1 - \ell)(g - 1 - 2\ell)}{\ell+1} \]
using the Hilbert polynomial of the canonical curve $C$. 
This formula also follows from \cite[Lem.\,1.3]{previouspaper}, by using instead the left hand side of the equality.
\end{remark}

\section{Gorenstein weak Fano toric surfaces} \label{section_gorensteinweakfano}

As before let $\Delta$ be a lattice polygon with two-dimensional interior $\Delta^{(1)}$. Let $\Sigma_\Delta$ denote the inner normal fan
to $\Delta$, and as in the proof of Lemma~\ref{fujitalemma} let $U(\Sigma_\Delta)$ be the set of primitive generators of its rays. The prime divisor associated
with $v \in U(\Sigma_\Delta)$ will again be denoted by $D_v$. For reasons 
that will become apparent in the next section, we are interested in situations where
the polygon $P_{-K_\Delta}$ associated with
the anticanonical divisor 
\[ -K_\Delta=\sum_{v\in U(\Sigma)}\,D_v \] 
on $X_\Delta$ is a lattice polygon. Using the criteria in \cite[Chapter 6]{coxlittleschenck} one sees that this holds if and only if $-K_\Delta$ is base-point free (i.e., nef) and Cartier. Since $-K_\Delta$ is always
big, we conclude that we are actually interested in the cases where $X_\Delta$ is a so-called \emph{Gorenstein weak Fano} toric surface.

Note that $P_{-K_{\Delta}}$ has one interior lattice point only, namely, the origin; therefore, in the Gorenstein weak Fano case it is a reflexive polygon.
Its dual polygon is the convex hull of the vectors $v \in U(\Sigma_\Delta)$, which is then also reflexive. It is not hard to
see that the argument works in both ways, i.e., a toric surface is Gorenstein weak Fano if and only if the convex hull of the primitive generators of
the rays of its fan is a reflexive polygon. 
Up to unimodular equivalence, there are $16$ reflexive polygons 
\cite[Prop.\,4.1]{nill}, so a toric surface is Gorenstein weak Fano if and only if its fan is a \emph{coherent crepant refinement} of the inner normal fan to one of
these $16$ polygons. That is, it is obtained by inserting a number of rays (possibly none) that
pass through a lattice point on the boundary of the dual polygon.
\begin{figure}[h]
\centering
\begin{tikzpicture}[scale=0.4]
  \draw[thick] (-12,0) -- (-13,1) -- (-14,-1) -- (-13,-1) -- (-12,0);
  \draw[fill] (-13,0) circle [radius=0.15];
  \draw[fill] (-12,0) circle [radius=0.15];
  \draw[fill] (-13,1) circle [radius=0.15];
  \draw[fill] (-14,-1) circle [radius=0.15];  
  \draw[fill] (-13,-1) circle [radius=0.15];  
  \draw[->] plot [smooth] coordinates {(-10,0) (-7,0.3) (-4,0)};
  \node at (-7,1){inner normal fan};
  \draw (0,0) -- (3,-1.5);
  \draw (0,0) -- (0,3);
  \draw (0,0) -- (-2,2);
  \draw (0,0) -- (-2,-2);
  \draw[dotted] (0,0) -- (2.5,0);
  \draw[dotted] (0,0) -- (-2.5,0);
  \draw[dotted] (0,0) -- (0,-2.5);
  \draw[dotted] (0,0) -- (2,-2);
  \draw[fill] (0,-1) circle [radius=0.15];
  \draw[fill] (1,-1) circle [radius=0.15];
  \draw[fill] (2,-1) circle [radius=0.15];
  \draw[fill] (1,0) circle [radius=0.15];
  \draw[fill] (0,1) circle [radius=0.15];
  \draw[fill] (-1,1) circle [radius=0.15];
  \draw[fill] (-1,0) circle [radius=0.15];
  \draw[fill] (-1,-1) circle [radius=0.15];
  \draw[thick] (-1,-1) -- (2,-1) -- (0,1) -- (-1,1) -- (-1,-1);
\end{tikzpicture}
\caption{Combinatorial characterization of the Gorenstein weak Fano property}
\label{combcharac}
\end{figure}
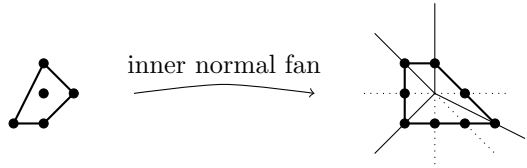
A similar criterion was proven to hold in any dimension by Nill~\cite[Prop.\,1.7]{nill}, to whom's paper we refer for more background.

The aim of the current section is to show that the Gorenstein weak Fano property enjoys a certain robustness.
\begin{lemma} \label{gwfrobustminimal}
  If $X_\Delta$ is Gorenstein weak Fano and $X \rightarrow X_\Delta$ is the minimal toric resolution of singularities, then
  also $X$ is Gorenstein weak Fano, and moreover  $P_{-K} = P_{-K_\Delta}$. 
\end{lemma} 
\noindent Here, as in the previous section, $K$ denotes the canonical divisor on $X$ obtained by taking minus the sum of all torus-invariant prime divisors.

\begin{proof}
Consider the maximal coherent crepant refinement of $\Sigma_\Delta$, obtained
by inserting a ray for \emph{each} lattice point on the boundary of the reflexive polygon obtained by taking the 
convex hull of $U(\Sigma)$. 
This clearly gives a resolution of singularities. Therefore
the fan $\Sigma$ of $X$ must be obtained from $\Sigma_\Delta$ by inserting a number of these rays (possibly none, possibly all). We conclude that $\Sigma$
is also a coherent crepant refinement of $\Sigma_\Delta$, and both claims follow.
\end{proof}

For our second robustness statement, we need the following notation.
Given a lattice polygon $\Delta$ with two-dimensional interior lattice polygon $\Delta^{(1)}$, then $\Delta^\text{max}$ is defined as the maximal lattice polygon $\Gamma$ (with respect to inclusion) satisfying $\Gamma^{(1)}=\Delta^{(1)}$. The polygon $\Delta^\text{max}$ can be obtained from $\Delta^{(1)}$ by moving out its edges over an integral distance $1$. Therefore each edge of $\Delta^\text{max}$ is parallel to an edge of $\Delta^{(1)}$, although the converse may fail, because it could happen
that an edge shrinks to length $0$. See 
\cite[\S2]{linearpencils} and the references therein for more background.

\begin{lemma} \label{gwfrobust}
If $X_\Delta$ is Gorenstein weak Fano, then
also $X_{\Delta^{(1)}}$ is Gorenstein weak Fano. Moreover, the latter property holds if and only if $X_{\Delta^\emph{max}}$ is Gorenstein weak Fano,
and in this case the normal fans to $\Delta^{(1)}$ and $\Delta^\emph{max}$ are the same.
\end{lemma}

\begin{proof}
We will rely on the following observation: let $X$ be a Gorenstein weak Fano projective toric surface, and let $X'$ be a toric blow-down of $X$, i.e., the toric surface obtained
by removing a certain number of rays from the fan defining $X$. Then $X'$ is also Gorenstein weak Fano. Indeed, if the primitive generators of the rays of a fan span a reflexive polygon, 
then this remains true after dropping some of these rays.

We first prove the last equivalence, namely that $X_{\Delta^{(1)}}$ is Gorenstein weak Fano if and only if the same is true for 
$X_{\Delta^\text{max}}$. As noted above, the inner normal fan to $\Delta^{(1)}$ is a subdivision of the inner normal fan to $\Delta^\text{max}$, which by the foregoing observation implies the `only if' part of the statement. As for the `if' part, assume that $\Delta^\text{max}$ is Gorenstein weak Fano. We will show that the subdivision is in fact trivial, i.e., the normal fans
to $\Delta^{(1)}$ and $\Delta^\text{max}$ are the same, from which the desired conclusion follows. Indeed, suppose that there is an edge 
$\tau \subseteq \Delta^{(1)}$ that disappears after moving out the edges, i.e., its length shrinks to $0$, and choose it
such that there is an adjacent edge $\tau'$ that does not disappear. Let $v$ be the vertex common to $\tau$ and $\tau'$. 
Using a unimodular transformation if needed we can assume that $\tau$ is supported on the line $y = 0$, that $v = (0,0)$, and that the next lattice point on $\tau'$ is $(-b,a)$ with $a \geq b \geq 1$.
The outward shifts of the supporting lines of $\tau$ and $\tau'$ meet in the point 
\[ w = \left( \frac{b-1}{a}, -1 \right), \]
which is necessarily a lattice point, hence $b=1$ and $w=(0,-1)$. Now let $\tau''$ be the first non-disappearing edge at the other side of $\tau$; note that it might a priori not be adjacent to it. Denote
its primitive inner normal vector by $(c,d)$, so that its supporting line is of the form $cx + dy = e$. Notice that $c \leq -1$ by convexity of $\Delta^{(1)}$, and moreover $e \leq c$ because $(1,0)$ is contained in the corresponding half-plane.
Now the outward shift of this line (defined by $cx + dy = e-1$) must also pass through $w$, leading to the identity
\[ d = - e + 1 > 1. \] 
This contradicts the being Gorenstein weak Fano of $\Delta^\text{max}$, because the convex hull of $(a,b)$, $(c,d)$ and the other primitive generators of the rays of its normal fan contains $(0,1)$ as an interior point.

As for the first implication, note that $\Delta$ is obtained from $\Delta^\text{max}$ by clipping off a number of vertices. We show that these vertices
can be glued back on, one by one, while preserving the Gorenstein weak Fano property. Remark that a vertex can only be clipped off if it is `smooth', meaning that a unimodular
transformation takes it to $(0,0)$ with the adjacent edges lining up with the coordinate axes: otherwise $\Delta^{(1)}$ would be affected. Up to changing the order of the coordinates,
the clipping then necessarily happens along the segment connecting $(0,1)$ and $(a,0)$ for some $a \geq 0$. We make a case distinction between
three removal types.
\begin{itemize}
  \item Type $1$: none of the adjacent edges was removed completely. This means that glueing back the vertex boils down to dropping a ray from the inner normal fan, which
  preserves the Gorenstein weak Fano property.
  \item Type $2$: exactly one of the adjacent edges was removed completely. Then the situation is among the ones depicted in Figure~\ref{edgeremovedcompletely}.
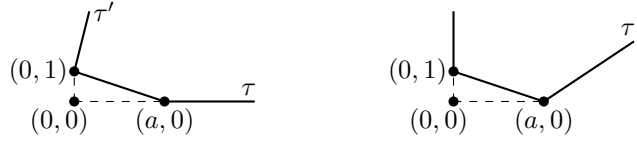
\begin{figure}[h]
\centering
\begin{tikzpicture}[scale=0.4]
  \draw[dashed] (0,0) -- (3,0);
  \draw[thick] (3,0) -- (6,0);
  \draw[dashed] (0,0) -- (0,1);
  \draw[thick] (0,1) -- (3,0);
  \draw[thick] (0,1) -- (0.5,3);
  \node at (-0.5,-0.65) {$(0,0)$};
  \node at (-1.2,1) {$(0,1)$};
  \node at (3,-0.65) {$(a,0)$};
  \draw[fill] (0,1) circle [radius=0.15];
  \draw[fill] (0,0) circle [radius=0.15];
  \draw[fill] (3,0) circle [radius=0.15];  
  \node at (5.8,0.4) {$\tau$};
  \node at (1,3) {$\tau'$};
\end{tikzpicture}
\qquad \qquad
\begin{tikzpicture}[scale=0.4]
  \draw[dashed] (0,0) -- (3,0);
  \draw[thick] (3,0) -- (6,2);
  \draw[dashed] (0,0) -- (0,1);
  \draw[thick] (0,1) -- (3,0);
  \draw[thick] (0,1) -- (0,3);
  \node at (-0.5,-0.65) {$(0,0)$};
  \node at (-1.2,1) {$(0,1)$};
  \node at (3,-0.65) {$(a,0)$};
  \draw[fill] (0,1) circle [radius=0.15];
  \draw[fill] (0,0) circle [radius=0.15];
  \draw[fill] (3,0) circle [radius=0.15];  
  \node at (5.8,2.4) {$\tau$};
 % \node at (1,3) {$\tau'$};
\end{tikzpicture}
\caption{The cases where exactly one edge is removed completely}
\label{edgeremovedcompletely}
\end{figure}
 One of the primitive generators of the rays of the inner normal fan to $\Delta$ is given by $(1,a)$. 
 
  In the first case, the primitive normal vector to $\tau'$ is of the form $(b,c)$ for some $c < 0$ and $b \geq 1$, where the latter inequality holds
  because $\tau'$ cannot be horizontal (otherwise
  $\Delta$ would have an empty interior). This means that $(1,0)$ belongs to the polygon
  spanned by the primitive generators, and therefore it stays reflexive upon replacement of $(1,a)$ by $(1,0)$, i.e., the Gorenstein weak Fano
  property is preserved when glueing back our vertex. (This reasoning shows that in fact $b=1$, because otherwise $(1,0)$ would be contained in the interior of the polygon spanned by the primitive generators, thereby violating that $X_\Delta$ is Gorenstein weak Fano.)
    
  In the second case we find $(1,0)$ among
  the primitive generators of the rays of the inner normal fan. If $a>2$ then by the Gorenstein weak Fano property all other primitive generators must belong to the triangle shown in Figure~\ref{triangularregion},
  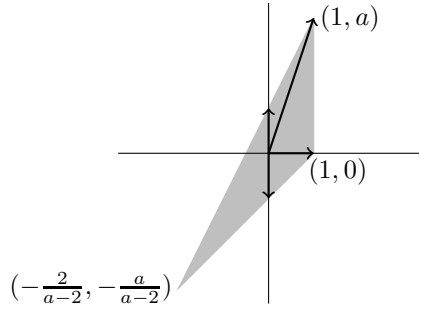
\begin{figure}[h]
\centering
\begin{tikzpicture}[scale=0.4]
  \draw[lightgray, fill=lightgray] (1.5,0) -- (1.5,4.5) -- (-3,-4.5) -- (1.5,0); 
  \draw[thin] (-5,0) -- (5,0);
  \draw[thin] (0,-5) -- (0,5);
  \draw[thick,->] (0,0) -- (1.5,0);
  \draw[thick,->] (0,0) -- (1.5,4.5);
  \draw[thick,->] (0,0) -- (0,1.5);
  \draw[thick,->] (0,0) -- (0,-1.5);
  \node at (2.3,-0.6) {$(1,0)$};
  \node at (2.7,4.5) {$(1,a)$};
  \node at (-5.9,-4.5) {$(-\frac{2}{a-2},-\frac{a}{a-2})$};
\end{tikzpicture}
\caption{Allowed region for the primitive generators in case $a > 2$}
\label{triangularregion}
\end{figure}
  because otherwise either $(0,1)$ or $(0,-1)$ would belong to the interior of the polygon they span. 
  If $a > 4$ then $2/(a-2) < 1$, so the triangular region cannot contain the primitive normal vector to $\tau$, which has to have a strictly negative first coordinate: a contradiction.
  If $a = 4$ then the primitive normal vector to $\tau$ is necessarily $(-1,-2)$, which gives a contradiction with the fact that $\Delta^{(1)}$ is two-dimensional. If $a=3$ one finds $(-1,-1)$, $(-1,-2)$ or $(-2,-3)$, each of which cases again yields a contradiction with the two-dimensionality of $\Delta^{(1)}$.
  
  If $a = 2$ then the region becomes the half-strip shown in Figure~\ref{halfstrip}, 
  \begin{figure}[h]
\centering
\begin{tikzpicture}[scale=0.4]
  \draw[lightgray, fill=lightgray] (1.5,0) -- (1.5,3) -- (-4.5,-3) -- (-1.5,-3) -- (1.5,0); 
  \draw[thin] (-5,0) -- (5,0);
  \draw[thin] (0,-3) -- (0,4);
  \draw[thick,->] (0,0) -- (1.5,0);
  \draw[thick,->] (0,0) -- (1.5,3);
  \draw[thick,->] (0,0) -- (0,1.5);
  \draw[thick,->] (0,0) -- (0,-1.5);
  \node at (2.3,-0.6) {$(1,0)$};
  \node at (2.7,3) {$(1,2)$};
\end{tikzpicture}
\caption{Allowed region for the primitive generators in case $a=2$}
\label{halfstrip}
\end{figure}
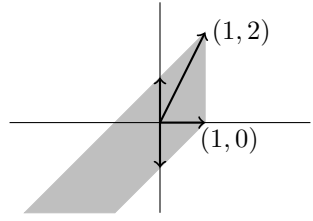
  where now the conclusion reads that the primitive normal vector to $\tau$ has a negative second coordinate (possibly zero): this means that $\Delta$ is contained in a vertical strip of width $2$, once again
  contradicting the fact that
  $\Delta^{(1)}$ is two-dimensional. Thus we conclude that $a = 1$, and a similar reasoning shows that the primitive normal vector to $\tau$ must be 
of the form $(b,c)$ for some $b < 0$ and $c \leq 1$. If $c < 1$ then we again run into a contradiction with the two-dimensionality of $\Delta^{(1)}$. Therefore $c=1$, but  
%    of the form $(b,1)$ for some $b < 0$.
%  But 
  this means that the polygon spanned by the primitive normal vectors contains $(0,1)$, and therefore the Gorenstein weak Fano property is preserved upon replacement of $(1,a) = (1,1)$ by $(0,1)$, i.e., upon
  glueing back our vertex.  
  \item Type $3$: the two adjacent edges are removed completely. Then the situation must be as depicted in Figure~\ref{twoedgesremovedcompletely}.
  \begin{figure}[h]
\centering
\begin{tikzpicture}[scale=0.4]
  \draw[dashed] (0,0) -- (3,0);
  \draw[thick] (3,0) -- (6,2);
  \draw[dashed] (0,0) -- (0,1);
  \draw[thick] (0,1) -- (3,0);
  \draw[thick] (0,1) -- (0.5,3);
  \node at (-0.5,-0.65) {$(0,0)$};
  \node at (-1.2,1) {$(0,1)$};
  \node at (3,-0.65) {$(a,0)$};
  \draw[fill] (0,1) circle [radius=0.15];
  \draw[fill] (0,0) circle [radius=0.15];
  \draw[fill] (3,0) circle [radius=0.15];  
  \node at (5.8,2.4) {$\tau$};
  \node at (1,3) {$\tau'$};
\end{tikzpicture}
\caption{The case where both edges are removed completely}
\label{twoedgesremovedcompletely}
\end{figure}
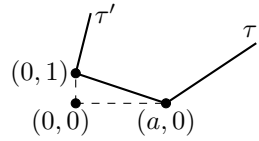
  This is very similar to before. In the cases where $a \geq 2$ one again obtains a contradiction, either with
  the Gorenstein weak Fano property or with the two-dimensionality of $\Delta^{(1)}$: the region in which the primitive normal vector to $\tau$ should be contained
  becomes even smaller. If $a=1$ then we find that the primitive normal vectors
  to $\tau$ and $\tau'$ are $(1,b)$ resp.\ $(b',1)$ for integers $b,b' < 0$, so we can replace $(1,1)$ by the pair $(0,1), (1,0)$, i.e., we can glue back our vertex.
\end{itemize}
This concludes the proof.
\end{proof}

One corollary is that, in the statement of Theorem~\ref{maintheorem}, the condition that $X_{\Delta^{(1)}}$ is Gorenstein weak
Fano can be replaced by $X_\Delta$ being Gorenstein weak Fano, although the resulting theorem is weaker.

\section{Constancy results} \label{section_constancyresults}

%\begin{definition}
%A lattice polygon $\Delta$ is called \emph{Gorenstein weak Fano} if and only if its normal fan is a coherent crepant refinement of the fan of a reflexive polygon. 
%\end{definition}

In this section we prove Theorem~\ref{maintheorem}. As before let $\Delta$ be a lattice polygon with two-dimensional interior $\Delta^{(1)}$, let
$f$ be as in \eqref{eq_Laurentpoly} and assume
that $C_f$ is a smooth hyperplane section of $X_\Delta$. 
We copy the set-up and notation from Section~\ref{sectionsixterms}, which we extend by writing
$V_\Gamma$ for the subspace of $k[x^{\pm 1}, y^{\pm 1}]$ consisting of those Laurent polynomials that are supported $\Gamma$, for any given lattice polygon $\Gamma$. Then \[ H^0(X,qL) = V_{q\Delta^{(1)}} \qquad \text{and} \qquad H^0(X,qL + K) = V_{(q\Delta^{(1)})^{(1)}} \] 
for all $q \geq 1$.

From~\eqref{sixterms} we obtain the following:

\begin{lemma} \label{whendoessumholdlemma}
  Letting $\Delta$ and $f \in k[x^{\pm 1}, y^{\pm 1}]$ be as above, for each $\ell=0,1, \ldots, g - 2$ the following assertions are equivalent:
  \begin{itemize}
    \item $a_\ell = b_\ell + c_\ell$, 
    \item $a_{g - 1 - \ell} = b_{g - 1 - \ell} + c_{g - 1 - \ell}$,
    \item $K_{\ell-1,1}(X;K,L) \stackrel{\mu_f}{\longrightarrow} K_{\ell-1,2}(X,L)$
  is the zero map.
  \end{itemize}
  Here $X, K, L$ and $a_0, a_1, \ldots, a_{g-2}$, $b_0, b_1, \ldots, b_{g-2}$, $c_0, c_1, \ldots, c_{g-2}$ are obtained from $\Delta$ and $f$ as in Section~\ref{sectionsixterms}.
\end{lemma}
\noindent Recall that $\mu_f$ denotes multiplication by $f$. Explicitly, this is the map induced by the vertical
maps (also denoted by $\mu_f$) of the commutative diagram
%For each $q \geq 0$ the space $H^0(X,qL)$ equals the $k$-vector subspace of $k[x^{\pm 1}, y^{\pm 1}]$ generated by the monomials $x^iy^j$ 
%for which $(i,j) \in q\Delta$. This follows from
%\cite[Prop.\,4.3.3]{coxlittleschenck} and Lemma~\ref{fujitalemma}. Similarly we find that $H^0(X,qL + K)$ is the space generated by $x^iy^j$ with
%$(i,j)$ in the interior of $q \Delta$.
%For the sake of abbreviation we denote these respective spaces by $V_{q\Delta}$ and $V_{(q \Delta)^{(1)}}$.\unsure{mss al in te voeren in computationeel stuk} \ 
\[ \begin{array}{ccccc} 
  &  & \bigwedge^{\ell-1} V_{\Delta^{(1)}} \otimes V_{\Delta^{(2)}} 
& \stackrel{\delta}{\longrightarrow} & \bigwedge^{\ell-2} V_{\Delta^{(1)}} \otimes V_{(2\Delta^{(1)})^{(1)}} \\
  & & \phantom{\mu_f} \downarrow \mu_f & & \phantom{\mu_f} \downarrow \mu_f \\
\bigwedge^\ell V_{\Delta^{(1)}} \otimes V_{\Delta^{(1)}} & \stackrel{\delta}{\longrightarrow} & \bigwedge^{\ell-1} V_{\Delta^{(1)}} \otimes V_{2\Delta^{(1)}} 
& \stackrel{\delta}{\longrightarrow} & \bigwedge^{\ell-2} V_{\Delta^{(1)}} \otimes V_{3\Delta^{(1)}}, \end{array}   \]
where the
$\delta$'s are the usual boundary morphisms
\[ v_1 \wedge v_2 \wedge v_3 \wedge v_4 \wedge \dots \otimes w \, \mapsto \, \sum_s (-1)^s v_1 \wedge v_2 \wedge v_3 \wedge v_4 \wedge \dots \wedge \widehat{v_s} \wedge \dots 
\otimes v_s w \]
($\widehat{v_s}$ means that $v_s$ is being omitted) and 
the $\mu_f$'s act like
\[
 v_1 \wedge v_2 \wedge v_3 \wedge v_4 \wedge \dots \otimes w \, \mapsto \, v_1 \wedge v_2 \wedge v_3 \wedge v_4 \wedge \dots \otimes fw,
\]
where $fw$ indeed ends up in the target space because $\Delta + \Delta^{(2)} \subseteq 2\Delta^{(1)}$ and $\Delta + (2\Delta^{(1)})^{(1)} \subseteq 3\Delta^{(1)}$.
%These induce the vector space morphism, also named $\mu_f$, from the statement of Lemma~\ref{whendoessumholdlemma}.

Then indeed
$K_{\ell-1,1}(X;K,L)$ is the kernel of the top row while
$K_{\ell-1,2}(X,L)$ is the cohomology in the middle of the bottom row. In view of Lemma~\ref{whendoessumholdlemma}, our aim is to find conditions under which $\mu_f = 0$ on the level of cohomology.
It is convenient to introduce a multiplication map for each
monomial $x^iy^j$ that is supported on $\Delta$. That is, for each $(i,j) \in \Delta \cap \mathbb{Z}^2$ we consider the morphism 
$\mu_{i,j} :  K_{\ell-1,1}(X;K,L) \rightarrow K_{\ell-1,2}(X,L)$ that is induced by
\[
 v_1 \wedge v_2 \wedge v_3 \wedge v_4 \wedge \dots \otimes w \, \mapsto \, v_1 \wedge v_2 \wedge v_3 \wedge v_4 \wedge \dots \otimes x^iy^jw. 
\]
Note that
\[ \mu_f = \sum_{(i,j) \in \Delta \cap \mathbb{Z}^2} c_{i,j} \mu_{i,j}. \]
In fact we even have
\begin{equation} \label{onlycoefficientsontheboundary}
 \mu_f = \sum_{(i,j) \in \partial \Delta \cap \mathbb{Z}^2} c_{i,j} \mu_{i,j}
\end{equation}
thanks to the following observation:
\begin{lemma} \label{interiordoesnotcontribute}
  If $(i,j) \in \Delta^{(1)}$ then $\mu_{i,j} = 0$ on the level of cohomology.
\end{lemma}
\begin{proof}
This follows from a well-known type of reasoning; see~\cite[(1.b.11)]{greenkoszul} or~\cite[Lem.\,2.19]{nagelaprodu}. Explicitly, if
\[ \alpha = \sum_r c_r v_{r,1} \wedge v_{r,2} \wedge \dots \wedge v_{r, \ell - 1} \otimes w_r \ \in \ {\bigwedge}^{\ell-1} \ V_{\Delta^{(1)}} \otimes V_{\Delta^{(2)}} \]
is in the kernel of $\delta$, then one verifies that $\mu_{i,j} (\alpha)$ is the coboundary of
\begin{equation} \label{preimagedelta}
 - \sum_r c_r x^iy^j \wedge v_{r,1} \wedge v_{r,2} \wedge \dots \wedge v_{r, \ell - 1} \otimes w_r \ \in \ {\bigwedge}^\ell \ V_{\Delta^{(1)}} \otimes V_{\Delta^{(1)}}
\end{equation}
and therefore vanishes on the level of cohomology.
\end{proof}

The above argument does not work for $(i,j) \in \partial \Delta$ because in that case
$x^i y^j \notin V_{\Delta^{(1)}}$ and therefore (\ref{preimagedelta}) may not be contained in $\bigwedge^\ell V_{\Delta^{(1)}} \otimes V_{\Delta^{(1)}}$.
However, the condition that $(i,j) \in \Delta^{(1)}$ can be relaxed:
\begin{lemma}
  If $(i,j) \in \Delta$ can be written as $(i_1, j_1) + (i_2,j_2)$ such that $(i_1,j_1) \in \Delta^{(1)}$ and
  $(i_2,j_2) + \Delta^{(2)} \subseteq \Delta^{(1)}$, then $\mu_{i,j} = 0$ on the level of cohomology.
\end{lemma}
\begin{proof}
In the above proof 
\[
 - \sum_r c_r x^{i_1}y^{j_1} \wedge v_{r,1} \wedge v_{r,2} \wedge \dots \wedge v_{r, \ell - 1} \otimes x^{i_2} y^{j_2} w_r
  \ \in \ {\bigwedge}^\ell \ V_{\Delta^{(1)}} \otimes V_{\Delta^{(1)}}
\]
serves as a replacement for (\ref{preimagedelta}).
%\unsure{Kunnen we hier de decompositie afhankelijk maken van de keuze van $w_r\in V_{\Delta^{(2)}}$?}
\end{proof}
\noindent This generalization of Lemma~\ref{interiordoesnotcontribute} can be seen in the context of~\cite[Rem.\ on p.\,134]{greenkoszul}, which hints at the existence of many ways to generalize~\cite[(1.b.11)]{greenkoszul}.

It is natural to try and take $(i_2,j_2) \in P_{-K}$, so that $x^{i_2}y^{j_2} \in H^0(X,-K)$. Indeed recall
that $V_{\Delta^{(2)}} = H^0(X,L + K)$ and $V_{\Delta^{(1)}} = H^0(X,L)$, so in this case we indeed have that $(i_2,j_2) + \Delta^{(2)} \subseteq \Delta^{(1)}$.
Such an appropriate decomposition of $(i,j) \in \Delta$ can be found only if
\begin{equation} \label{intheproduct}
% x^iy^j \in H^0(X,-K) \cdot H^0(X,L).
x^iy^j \in \left \{ \, \varphi \psi \, | \, \varphi \in H^0(X,-K) \text{ and } \psi \in H^0(X,L) \, \right\}.
\end{equation}
Often $H^0(X,-K)$ consists of the constant functions only, i.e., $P_{-K} \cap \mathbb{Z}^2 = \{(0,0)\}$, in which case (\ref{intheproduct}) is impossible as soon as 
$(i,j) \in \partial \Delta$.
On the other hand, if the right hand side of (\ref{intheproduct}) generates all of $H^0(X,D_f)$, or equivalently, if the map 
\begin{equation} \label{tensormap} H^0(X,-K) \otimes H^0(X,L) \to H^0(X,D_f) \end{equation} 
is surjective, then we can conclude that all $\mu_{i,j}$'s are zero on cohomology, and therefore the same is true for $\mu_f$.
%In this case, Lemma \ref{whendoessumholdlemma} then implies that the desired formula $a_\ell=b_\ell+c_\ell$ holds for all $\ell=1,\ldots,g-3$. 
%In particular, the graded Betti table of $C$ is independent of the coefficients of $f$. 
%\begin{lemma}
%If the map \begin{equation} \label{tensormap} H^0(X,-K) \otimes H^0(X,L) \to H^0(X,D_f) \end{equation} is surjectice, then the formula $a_\ell=b_\ell+c_\ell$ holds for all $\ell=1,\ldots,g-3$. In particular, the Betti table of $C$ is independent from the coefficients of $f$. 
%\end{lemma}

We are ready to prove our main result, essentially by establishing that~\eqref{tensormap} is indeed surjective in the Gorenstein weak Fano case. Notice that this is, in fact, immediate for $X = \PP^2$ and $X = \PP^1 \times \PP^1$, i.e., for the cases \emph{(a)} and \emph{(b)} that were highlighted in the introduction.

\begin{proof}[Proof of Theorem~\ref{maintheorem}]
 We will assume that $\Delta = \Delta^\text{max}$, i.e., $\Delta$ is the maximal
polygon having $\Delta^{(1)}$ as its interior.
This is not a restriction: as soon as $C_f \subseteq X_{\Delta}$ is a smooth hyperplane section, this is also the case for 
the Zariski closure of $\varphi_{\Delta^\text{max}}(U_f)$ viewed inside $X_{\Delta^\text{max}}$, as explained in~\cite[\S4]{linearpencils}.
Moreover, the statement of Theorem~\ref{maintheorem} only involves $\Delta^{(1)}$, which is left unaffected.

We first deal with the case where $X_{\Delta^{(1)}}$ is Gorenstein weak Fano,
 which by Lemma~\ref{gwfrobust} holds if and only if $X_\Delta$ is Gorenstein weak Fano (because of our
 assumption that $\Delta$ is maximal). By Lemma~\ref{gwfrobustminimal} then also $X$ is Gorenstein weak Fano, and moreover
 $P_{-K} = P_{-K_\Delta}$.
  Now note that
 \[ P_{-K} + \Delta^{(1)} = \Delta. \]
 Indeed, the inclusion $\subseteq$ is obvious, while for the other inclusion it is enough to prove that each 
 vertex $m$ of $\Delta$ is in $P_{-K} + \Delta^{(1)}$. Let $v,w$ be consecutive elements of $U(\Sigma)$ such that $m=L(v,a_v)\cap L(w,a_w)$. By the proof of Lemma \ref{fujitalemma} we know that $m_1=L(v,a_v-1)\cap L(w,a_w-1)\in\Delta^{(1)}$, hence $m=m_0+m_1\in P_{-K} +\Delta^{(1)}$ with $m_0=L(v,1)\cap L(w,1)$.

 But then also
 \[ (P_{-K} \cap \mathbb{Z}^2) \  + \ ( \Delta^{(1)} \cap \mathbb{Z}^2) \ = \ \Delta \cap \mathbb{Z}^2 \]
 by \cite[Thm\,1.1]{HNPS}, 
 because
 the inner normal fan to $P_{-K} = P_{-K_\Delta}$ coarsens that of $\Delta^{(1)}$. Indeed, it obviously coarsens
 the inner normal fan to $\Delta$, which by Lemma~\ref{gwfrobust} is equal to the inner normal fan to $\Delta^{(1)}$.
 But this precisely means that \eqref{tensormap} is surjective, so the maps $\mu_f$ are all trivial on the level of cohomology, 
 and the conclusion follows from Lemma~\ref{whendoessumholdlemma}.
 
 As for the other case where $ | \partial \Delta^{(1)} \cap \ZZ^2| \geq g/2 + 1$, the maps $\mu_f$ 
 are trivial for a much simpler reason, namely because the dimension $c_\ell$ of the domain or the dimension $c_{g-1-\ell}$ of the 
 codomain (or both) is zero. This in turn
 follows
 from a result due to Hering and Schenck, stating that 
 $\min \{ \, \ell \, | \, c_{g-\ell} \neq 0 \, \} = | \partial \Delta^{(1)} \cap \ZZ^2 |$; see~\cite[Thm.\,IV.20]{heringphd} or~\cite{schenck}.
\end{proof}

We believe that the sum formula $a_\ell = b_\ell + c_\ell$ holds for a considerably larger class of lattice polygons, although there are counterexamples (if there were not, then
this would have negative consequences for Green's canonical syzygy conjecture, as explained in Remark~\ref{countergreenremark} in the next section).
The smallest counterexample we found lives in genus $g = 12$. Namely, consider $f = x^6 + y^2 + x^2y^6$ along with its Newton polygon $\Delta = \conv \{ (0,2), (6,0), (2,6) \}$.
A computer calculation along the lines of~\cite{canonical} shows that the graded Betti table of the canonical model of $C_f$ is given by
\[ \begin{array}{r|ccccccccccc}
  & 0 & 1 & 2 & 3 & 4 & 5 & 6 & 7 & 8 & 9 & 10 \\
\hline
0 & 1 & 0 & 0 & 0 & 0 & 0 & 0 & 0 & 0 & 0 & 0 \\
1 & 0 & 45 & 231 & 550 & 693 & 399 & 69 & 0  & 0 & 0 & 0 \\
2 & 0 & 0 & 0 & 0 & 69 & 399 & 693 & 550 & 231 & 45 & 0 \\ 
3 & 0 & 0 & 0 & 0 & 0 & 0 & 0 & 0 & 0 & 0 & 1 \\   
\end{array}
\]
while that of $X_{\Delta^{(1)}}$ is given by
\[ \begin{array}{r|cccccccccc}
  & 0 & 1 & 2 & 3 & 4 & 5 & 6 & 7 & 8 & 9 \\
\hline
0 & 1 & 0 & 0 & 0 & 0 & 0 & 0 & 0 & 0 & 0 \\
1 & 0 & 39 & 186 & 414 & 504 & 295 & 69 & 0  & 0 & 0 \\
2 & 0 & 0 & 0 & 0 & 1 & 105 & 189 & 136 & 45 & \phantom{.}6. \\ 
\end{array}
\]
Here one sees that the exact sequence \eqref{sixterms} for $\ell = 5$ reads:
\[ 0 \rightarrow 295 \rightarrow 399 \rightarrow 105 \stackrel{\mu_f}{\longrightarrow} 1 \rightarrow 69 \rightarrow 69 \rightarrow 0. \]
So $\mu_f$ is not trivial in this case, but rather surjective onto its one-dimensional codomain.

Another natural question is whether it is true in general that the graded Betti table of $C$ is independent of the coefficients of $f$, even if the sum formula does not hold.
In general one has for each $\ell = 1, \dots, g-3$ that
\[ a_\ell = b_\ell + c_\ell - \dim \im \mu_f \]
and constancy holds if and only if $\dim \im \mu_f$ depends on $\Delta$ and $\ell$ only. From \eqref{onlycoefficientsontheboundary} it follows
that, at least, there is no dependence on the coefficients $c_{i,j}$ that are supported on $\Delta^{(1)}$. In other words, only the coefficients that are supported on the boundary might a priori matter. 
A consequence of this observation is that constancy of the graded Betti table holds for primitive lattice triangles, i.e., lattice triangles without lattice points on the boundary, except for the three vertices.
Indeed, using a transformation of the form $f \leftarrow \gamma f(\alpha x, \beta y)$, with $\alpha, \beta, \gamma \in k^\ast$, one can always
arrange that the three coefficients supported on the vertices $(i_1,j_1), (i_2,j_2), (i_3,j_3)$ are all $1$. This means that $a_\ell = b_\ell + c_\ell - \dim \im \left( \mu_{i_1,j_1} + \mu_{i_2,j_2} + \mu_{i_3,j_3} \right)$, regardless
of the coefficients of $f$.

\section{Connections with Green's conjecture} \label{section_greenconnections}

In this section we elaborate the details of the announcements made in the last paragraph of the introduction, i.e., we deduce new cases 
of Green's canonical syzygy conjecture from known cases of
Conjecture~\ref{conj_CCDL} on syzygies of projectively embedded toric surfaces, and vice versa.
As a preliminary remark, note that if $\Delta^{(1)} \cong \Upsilon$ then Conjecture~\ref{conj_CCDL} is tautologically true, while Green's conjecture is known to hold
for curves of genus $g = N_\Upsilon = 4$. Therefore we can ignore this case in the proofs below.

We first prove Lemma~\ref{lemma2conj}, establishing a connection between both conjectures.

\begin{proof}[Proof of Lemma~\ref{lemma2conj}]
The right hand sides of the equalities in Conjecture~\ref{conj_Green} and Conjecture~\ref{conj_CCDL} agree by Theorem \ref{cliffindexthm}. Let us denote this common quantity by $\gamma$. 

First assume that Conjecture \ref{conj_Green} holds for some smooth irreducible hyperplane section $C_f \subseteq X_\Delta$. To deduce Conjecture \ref{conj_CCDL} for $X_{\Delta^{(1)}}$, it suffices to prove that $b_{g-(\gamma-1)}=0$. This follows from the fact that $a_{g-(\gamma-1)}=0$, along
with the exact sequence \eqref{sixterms} for $\ell=g-(\gamma-1)$.

For the other implication we need to show that $a_{g-(\gamma-1)}=0$. Since $b_{g-(\gamma-1)}=0$ by assumption, 
thanks to \eqref{sixterms} it suffices to show that $c_{g-(\gamma-1)}=0$. But this follows from Hering and Schenck's aforementioned result~\cite[Thm.\,IV.20]{heringphd}
that $\min\{\ell\,|\,c_{g-\ell}\neq 0\}=|\partial \Delta^{(1)}\cap\mathbb{Z}^2|$. Because of the stated inequality, we have that 
$$\gamma-1\leq \lw(\Delta^{(1)})+1\leq |\partial \Delta^{(1)}\cap \mathbb{Z}^2|-1,$$ hence indeed $c_{g-(\gamma-1)}=0$. 
\end{proof} 

\subsection*{New cases of Green's conjecture from known cases of Conjecture~\ref{conj_CCDL}}

We use Lemma~\ref{lemma2conj} to prove Green's conjecture for smooth curves of genus at most $32$ or Clifford index at most $6$ on arbitrary toric surfaces (Theorem~\ref{thm_greenlowgenuslowcliff}). Our key tool is the following lemma, showing that within these ranges, the additional condition that $|\partial\Delta^{(1)}\cap\ZZ^2| \geq \lw(\Delta^{(1)})+2$ is not a concern:

\begin{lemma}
Let $\Delta$ be a two-dimensional lattice polygon and assume that $\Delta^{(1)} \not \cong \Upsilon$. 
If $N_{\Delta^{(1)}} \leq 32$ or if $\lw(\Delta^{(1)}) \leq 6$ then 
$|\partial\Delta^{(1)}\cap\ZZ^2| \geq \lw(\Delta^{(1)})+2.$ 
\end{lemma}

\begin{proof}
The cases where $N_{\Delta^{(1)}} \leq 32$ are covered by an explicit computational verification, again using the data from~\cite{movingout}: as it turns out, up to unimodular equivalence the only two-dimensional interior lattice polygon $\Delta^{(1)}$ within this range that does \emph{not} satisfy 
the stated inequality is $\Delta^{(1)} = \Upsilon$. 

As for the cases where $\lw(\Delta^{(1)}) \leq 6$, it suffices to assume that $\lw(\Delta^{(1)}) \geq 3$, because the other cases are easy (for instance, it is an exercise to show that up to unimodular equivalence, the only two-dimensional interior lattice polygons whose boundary contains exactly three lattice points are $\Sigma$ and $\Upsilon$).
We prove the statement by contradiction, so assume that $$|\partial\Delta^{(1)}\cap\ZZ^2|\leq \lw(\Delta^{(1)})+1.$$ Moreover, we may assume that 
$\Delta^{(1)}$ lies in the horizontal strip $\RR\times[1,\lw(\Delta^{(1)})+1]$ and that it is not unimodularly equivalent to $k\Sigma$ for some $k$ (since the lemma holds for such polygons), hence $\Delta\subseteq \RR\times[0,\lw(\Delta^{(1)})+2]$.  
Denote by $\mathcal{L}$ and $\mathcal{R}$ the lattice points of respectively the left hand side boundary and the right hand side boundary of 
$\Delta^{(1)}$. After a reflection over a vertical axis if needed, we may assume that $|\mathcal{L}|\leq | \mathcal{R}|$. Since 
$$|\mathcal{L}\cup \mathcal{R}|\leq \lw(\Delta^{(1)})+1 \quad \text{and} \quad |\mathcal{L}\cap \mathcal{R}|\leq 2,$$ 
we get that $$|\mathcal{L}| \leq \floor{\frac{\lw(\Delta^{(1)})+3}{2}}.$$ 
The above inequality is very restrictive. After an exhaustive search, we can list all possibilities for $\mathcal{L}$ up to a unimodular transformation (to be more precise, a reflection over a horizontal axis and/or a horizontal shearing) under our assumption that $\lw(\Delta^{(1)})\in\{3,4,5,6\}$. Hereby, we also use that $\Delta^{(1)}$ is an interior lattice polygon (so the corners at vertices have to be sufficiently `good') and that 
$\Delta$ is contained in $\RR\times[0,\lw(\Delta^{(1)})+2]$. Figure~\ref{boundary} shows all these possibilities: the black edges represent the left hand side boundary of $\Delta^{(1)}$ and the blue edges the induced part of the boundary of $\Delta$. 
\begin{figure}[h]
\centering
\begin{tikzpicture}[scale=0.4]
  \draw[lightgray] (0,0) grid (23,15);

  % ONDER 1  
  \draw[blue] (1,1) -- (3,7);
  \draw[blue,fill=blue] (1,1) circle [radius=0.15]; 
  \draw[blue,fill=blue] (2,4) circle [radius=0.15];
  \draw[blue,fill=blue] (3,7) circle [radius=0.15];
  \draw (2,2) -- (2,3) -- (3,6) -- (4,7);
  \draw[fill] (2,2) circle [radius=0.15]; 
  \draw[fill] (2,3) circle [radius=0.15];
  \draw[fill] (3,6) circle [radius=0.15];   
  \draw[fill] (4,7) circle [radius=0.15];   
  
  % ONDER 2 
  \draw[blue] (7,0) -- (9,8);
  \draw[blue,fill=blue] (7,0) circle [radius=0.15]; 
  \draw[blue,fill=blue] (8,4) circle [radius=0.15];
  \draw[blue,fill=blue] (9,8) circle [radius=0.15];
  \draw (8,1) -- (8,3) -- (9,7);
  \draw[fill] (8,1) circle [radius=0.15]; 
  \draw[fill] (8,2) circle [radius=0.15];
  \draw[fill] (8,3) circle [radius=0.15];   
  \draw[fill] (9,7) circle [radius=0.15]; 
  
  % ONDER 3 
  \draw[blue] (13,0) -- (17,8);
  \draw[blue,fill=blue] (13,0) circle [radius=0.15]; 
  \draw[blue,fill=blue] (14,2) circle [radius=0.15];
  \draw[blue,fill=blue] (15,4) circle [radius=0.15];
  \draw[blue,fill=blue] (16,6) circle [radius=0.15];
  \draw[blue,fill=blue] (17,8) circle [radius=0.15];
  \draw (14,1) -- (17,7);
  \draw[fill] (14,1) circle [radius=0.15]; 
  \draw[fill] (15,3) circle [radius=0.15];
  \draw[fill] (16,5) circle [radius=0.15];   
  \draw[fill] (17,7) circle [radius=0.15];

  % ONDER 4 
  \draw[blue] (19,0) -- (21,6) -- (22,8);
  \draw[blue,fill=blue] (19,0) circle [radius=0.15]; 
  \draw[blue,fill=blue] (20,3) circle [radius=0.15];
  \draw[blue,fill=blue] (21,6) circle [radius=0.15];
  \draw[blue,fill=blue] (22,8) circle [radius=0.15];
  \draw (20,1) -- (20,2) -- (21,5) -- (22,7);
  \draw[fill] (20,1) circle [radius=0.15]; 
  \draw[fill] (20,2) circle [radius=0.15];
  \draw[fill] (21,5) circle [radius=0.15];   
  \draw[fill] (22,7) circle [radius=0.15];

  % BOVEN 1  
  \draw[blue] (0,10) -- (2,14);
  \draw[blue,fill=blue] (0,10) circle [radius=0.15]; 
  \draw[blue,fill=blue] (1,12) circle [radius=0.15];
  \draw[blue,fill=blue] (2,14) circle [radius=0.15];
  \draw (1,10) -- (1,11) -- (2,13);
  \draw[fill] (1,10) circle [radius=0.15]; 
  \draw[fill] (1,11) circle [radius=0.15];
  \draw[fill] (2,13) circle [radius=0.15];   

  % BOVEN 2
  \draw[blue] (5,9) -- (8,15);
  \draw[blue,fill=blue] (5,9) circle [radius=0.15]; 
  \draw[blue,fill=blue] (6,11) circle [radius=0.15];
  \draw[blue,fill=blue] (7,13) circle [radius=0.15];
  \draw[blue,fill=blue] (8,15) circle [radius=0.15];
  \draw (6,10) -- (8,14);
  \draw[fill] (6,10) circle [radius=0.15]; 
  \draw[fill] (7,12) circle [radius=0.15];
  \draw[fill] (8,14) circle [radius=0.15];  
  
  % BOVEN 3
  \draw[blue] (10,9) -- (12,15);
  \draw[blue,fill=blue] (10,9) circle [radius=0.15]; 
  \draw[blue,fill=blue] (11,12) circle [radius=0.15];
  \draw[blue,fill=blue] (12,15) circle [radius=0.15];
  \draw (11,10) -- (11,11) -- (12,14);
  \draw[fill] (11,10) circle [radius=0.15]; 
  \draw[fill] (11,11) circle [radius=0.15];
  \draw[fill] (12,14) circle [radius=0.15];
  
  % BOVEN 4
  \draw[blue] (15,9) -- (17,15);
  \draw[blue,fill=blue] (15,9) circle [radius=0.15]; 
  \draw[blue,fill=blue] (16,12) circle [radius=0.15];
  \draw[blue,fill=blue] (17,15) circle [radius=0.15];
  \draw (16,9) -- (16,11) -- (17,14);
  \draw[fill] (16,9) circle [radius=0.15]; 
  \draw[fill] (16,10) circle [radius=0.15]; 
  \draw[fill] (16,11) circle [radius=0.15];
  \draw[fill] (17,14) circle [radius=0.15];
  
  % BOVEN 5
  \draw[blue] (20,9) -- (23,15);
  \draw[blue,fill=blue] (20,9) circle [radius=0.15]; 
  \draw[blue,fill=blue] (21,11) circle [radius=0.15];
  \draw[blue,fill=blue] (22,13) circle [radius=0.15];
  \draw[blue,fill=blue] (23,15) circle [radius=0.15];
  \draw (21,9) -- (21,10) -- (23,14);
  \draw[fill] (21,9) circle [radius=0.15]; 
  \draw[fill] (21,10) circle [radius=0.15]; 
  \draw[fill] (22,12) circle [radius=0.15];
  \draw[fill] (23,14) circle [radius=0.15];
   
\end{tikzpicture}
\caption{Possibilities for $\mathcal{L}$}
\label{boundary}
\end{figure}
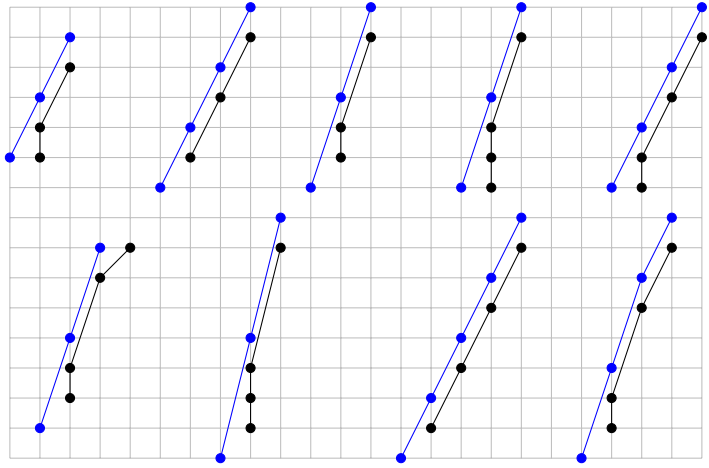
Note that $|\mathcal{L}|$ meets the upper bound $\floor{\frac{\lw(\Delta^{(1)})+3}{2}}$ in all these cases, hence 
$$|\mathcal{L}|\leq |\mathcal{R}| \leq \ceil{\frac{\lw(\Delta^{(1)})+3}{2}} = \begin{cases} |\mathcal{L}| & \text{if } 
\lw(\Delta^{(1)})\in\{3,5\}\\ |\mathcal{L}|+1 & \text{if } \lw(\Delta^{(1)})\in\{4,6\}\end{cases},$$ 
so either $|\mathcal{R}|=| \mathcal{L}|$ or $| \mathcal{R}|=|\mathcal{L}|+1$ and $\lw(\Delta^{(1)})\in\{4,6\}$. 

First assume that $|\mathcal{R}|=|\mathcal{L}|$, hence also $\mathcal{R}$ is listed in Figure \ref{boundary} up to a unimodular transformation. If $\lw(\Delta^{(1)})\in\{3,5\}$, then $|\mathcal{L}\cap\mathcal{R}|=2$, so there are only two lattice points of $\Delta^{(1)}$ on the lines at height $1$ and $\lw(\Delta^{(1)})+1$. Given $\mathcal{L}$, this leaves only a couple of possibilities for $\mathcal{R}$ and $\Delta^{(1)}$. For each of these, the lattice width is smaller then assumed, a contradiction. If $\lw(\Delta^{(1)})\in\{4,6\}$, then 
either $|\mathcal{L}\cap\mathcal{R}|=2$ and we can proceed as before, or $|\mathcal{L}\cap\mathcal{R}|=1$ and there are three lattice points of $\Delta^{(1)}$ on the lines at height $1$ and $\lw(\Delta^{(1)})+1$. Again this suffices to give a list of all possible polygons $\Delta^{(1)}$, none of which has the correct lattice width, leading to a contradiction.  

Finally, assume that $| \mathcal{R}|=| \mathcal{L}|+1$ and $\lw(\Delta^{(1)})\in\{4,6\}$. In this case, we have that $|\mathcal{L}\cap\mathcal{R}|=2$, hence there are only two lattice points of $\Delta^{(1)}$ on the lines at height $1$ and $\lw(\Delta^{(1)})+1$. Given a possible $\mathcal{L}$, convexity considerations provide a region in which $\Delta^{(1)}$ must strictly lie (here, `strict' means that no boundary point of the region is in $\Delta^{(1)}$). In Figure \ref{region}, the boundaries of the regions are indicated by the blue normal/dashed lines, so each lattice point of $\Delta^{(1)}$ must be one of the black dots. 
\begin{figure}[h]
\centering
\begin{tikzpicture}[scale=0.4]
  \draw[lightgray] (0,0) grid (27,14);
  
  %POLYGON ONDER 1
  \draw[blue] (0,0) -- (2,8);
  \draw[blue,fill=blue] (0,0) circle [radius=0.15]; 
  \draw[blue,fill=blue] (1,4) circle [radius=0.15];
  \draw[blue,fill=blue] (2,8) circle [radius=0.15];
  \draw (1,1) -- (1,3) -- (2,7);
  \draw[fill] (1,1) circle [radius=0.15]; 
  \draw[fill] (1,2) circle [radius=0.15]; 
  \draw[fill] (1,3) circle [radius=0.15];
  \draw[fill] (2,7) circle [radius=0.15];
  % andere punten
  \draw[fill] (2,2) circle [radius=0.15]; 
  \draw[fill] (2,3) circle [radius=0.15]; 
  \draw[fill] (2,4) circle [radius=0.15];
  \draw[fill] (2,5) circle [radius=0.15];
  \draw[fill] (2,6) circle [radius=0.15]; 
  \draw[fill] (3,2) circle [radius=0.15];
  \draw[fill] (3,3) circle [radius=0.15];  
  \draw[fill] (3,4) circle [radius=0.15];
  \draw[fill] (3,5) circle [radius=0.15];
  \draw[fill] (3,6) circle [radius=0.15];
  \draw[fill] (4,3) circle [radius=0.15];
  \draw[fill] (4,4) circle [radius=0.15];
  \draw[fill] (4,5) circle [radius=0.15];
  \draw[fill] (5,3) circle [radius=0.15];
  \draw[fill] (5,4) circle [radius=0.15];
  % lijnen
  \draw[blue,dashed] (0,0) -- (6.66,3.33) -- (2,8);
  
  %POLYGON ONDER 2
  \draw[blue] (9,0) -- (13,8);
  \draw[blue,fill=blue] (9,0) circle [radius=0.15]; 
  \draw[blue,fill=blue] (10,2) circle [radius=0.15];
  \draw[blue,fill=blue] (11,4) circle [radius=0.15];
  \draw[blue,fill=blue] (12,6) circle [radius=0.15];
  \draw[blue,fill=blue] (13,8) circle [radius=0.15];
  \draw (10,1) -- (13,7);
  \draw[fill] (10,1) circle [radius=0.15]; 
  \draw[fill] (11,3) circle [radius=0.15]; 
  \draw[fill] (12,5) circle [radius=0.15];
  \draw[fill] (13,7) circle [radius=0.15];
  % andere punten
  \draw[fill] (11,2) circle [radius=0.15]; 
  \draw[fill] (12,2) circle [radius=0.15]; 
  \draw[fill] (12,3) circle [radius=0.15];
  \draw[fill] (12,4) circle [radius=0.15];
  \draw[fill] (13,3) circle [radius=0.15]; 
  \draw[fill] (13,4) circle [radius=0.15];
  \draw[fill] (13,5) circle [radius=0.15];  
  \draw[fill] (13,6) circle [radius=0.15];
  \draw[fill] (14,3) circle [radius=0.15];
  \draw[fill] (14,4) circle [radius=0.15];
  \draw[fill] (14,5) circle [radius=0.15];
  \draw[fill] (14,6) circle [radius=0.15];
  \draw[fill] (15,4) circle [radius=0.15];
  \draw[fill] (15,5) circle [radius=0.15];
  \draw[fill] (16,4) circle [radius=0.15];
  % lijnen
  \draw[blue,dashed] (9,0) -- (17,4) -- (13,8); 

  %POLYGON ONDER 3
  \draw[blue] (19,0) -- (21,6) -- (22,8);
  \draw[blue,fill=blue] (19,0) circle [radius=0.15]; 
  \draw[blue,fill=blue] (20,3) circle [radius=0.15];
  \draw[blue,fill=blue] (21,6) circle [radius=0.15];
  \draw[blue,fill=blue] (22,8) circle [radius=0.15];
  \draw (20,1) -- (20,2) -- (21,5) -- (22,7);
  \draw[fill] (20,1) circle [radius=0.15]; 
  \draw[fill] (20,2) circle [radius=0.15]; 
  \draw[fill] (21,5) circle [radius=0.15];
  \draw[fill] (22,7) circle [radius=0.15];
  % andere punten
  \draw[fill] (21,2) circle [radius=0.15]; 
  \draw[fill] (21,3) circle [radius=0.15]; 
  \draw[fill] (21,4) circle [radius=0.15];
  \draw[fill] (22,2) circle [radius=0.15];
  \draw[fill] (22,3) circle [radius=0.15]; 
  \draw[fill] (22,4) circle [radius=0.15];
  \draw[fill] (22,5) circle [radius=0.15];  
  \draw[fill] (22,6) circle [radius=0.15];
  \draw[fill] (23,3) circle [radius=0.15];
  \draw[fill] (23,4) circle [radius=0.15];
  \draw[fill] (23,5) circle [radius=0.15];
  \draw[fill] (23,6) circle [radius=0.15];
  \draw[fill] (24,3) circle [radius=0.15];
  \draw[fill] (24,4) circle [radius=0.15];
  \draw[fill] (24,5) circle [radius=0.15];
  \draw[fill] (25,4) circle [radius=0.15];
  % lijnen
  \draw[blue,dashed] (19,0) -- (26.33,3.66) -- (22,8); 

  %POLYGON BOVEN 1
  \draw[blue] (6,8) -- (9,14);
  \draw[blue,fill=blue] (6,8) circle [radius=0.15]; 
  \draw[blue,fill=blue] (7,10) circle [radius=0.15];
  \draw[blue,fill=blue] (8,12) circle [radius=0.15];
  \draw[blue,fill=blue] (9,14) circle [radius=0.15];
  \draw (7,9) -- (8,11) -- (9,13);
  \draw[fill] (7,9) circle [radius=0.15]; 
  \draw[fill] (8,11) circle [radius=0.15]; 
  \draw[fill] (9,13) circle [radius=0.15];
  % andere punten
  \draw[fill] (8,10) circle [radius=0.15]; 
  \draw[fill] (9,10) circle [radius=0.15]; 
  \draw[fill] (9,11) circle [radius=0.15];
  \draw[fill] (9,12) circle [radius=0.15];
  \draw[fill] (9,13) circle [radius=0.15]; 
  \draw[fill] (10,11) circle [radius=0.15];
  \draw[fill] (10,12) circle [radius=0.15];  
  \draw[fill] (11,11) circle [radius=0.15];
  % lijnen
  \draw[blue,dashed] (6,8) -- (12,11) -- (9,14);

  %POLYGON BOVEN 2
  \draw[blue] (17,8) -- (19,14);
  \draw[blue,fill=blue] (17,8) circle [radius=0.15]; 
  \draw[blue,fill=blue] (18,11) circle [radius=0.15];
  \draw[blue,fill=blue] (19,14) circle [radius=0.15];
  \draw (18,9) -- (18,10) -- (19,13);
  \draw[fill] (18,9) circle [radius=0.15]; 
  \draw[fill] (18,10) circle [radius=0.15]; 
  \draw[fill] (19,13) circle [radius=0.15];
  % andere punten
  \draw[fill] (19,10) circle [radius=0.15]; 
  \draw[fill] (19,11) circle [radius=0.15]; 
  \draw[fill] (19,12) circle [radius=0.15]; 
  \draw[fill] (20,10) circle [radius=0.15];
  \draw[fill] (20,11) circle [radius=0.15];
  \draw[fill] (20,12) circle [radius=0.15]; 
  \draw[fill] (21,11) circle [radius=0.15];
  % lijnen
  \draw[blue,dashed] (17,8) -- (22.33,10.66) -- (19,14);

\end{tikzpicture}

% \includegraphics[width=0.5\textwidth]{region4.pdf} \\
 % \includegraphics[width=0.9\textwidth]{region6.pdf}
%\subfloat{\includegraphics[width=0.35\textwidth]{region4.pdf}}
%\hspace{0.02\textwidth}
%\subfloat{\includegraphics[width=0.60\textwidth]{region6.pdf}}
\caption{Regions for $\Delta^{(1)}$}
\label{region}
\end{figure}
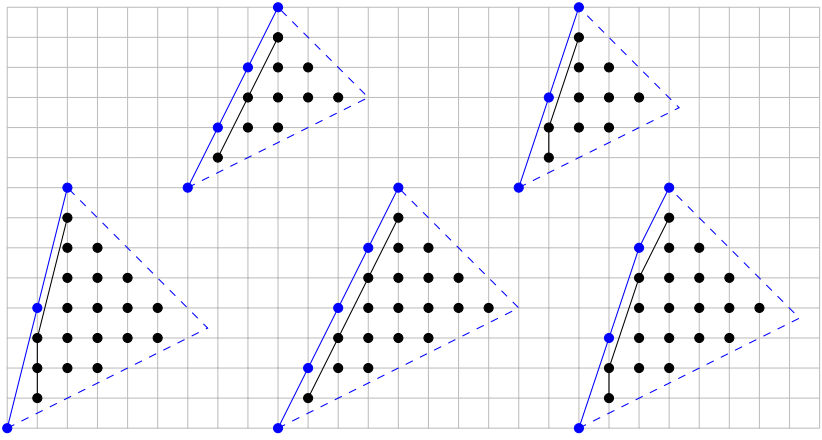

\noindent So either $\Delta^{(1)}$ is too small (i.e., it doesn't have the correct lattice width), or it is of the form $k\Upsilon$ with $k\in\{2,3\}$, but then $|\partial\Delta^{(1)}\cap\ZZ^2|=3k\geq 2k+2=\lw(\Delta^{(1)})+2$, a contradiction. 
\end{proof}

\begin{proof}[Proof of Theorem~\ref{thm_greenlowgenuslowcliff}]
It suffices to prove
the conjecture for curves of the form $C_f \subseteq X_\Delta$.
Now, as mentioned, our Conjecture \ref{conj_CCDL} has been verified for all interior lattice polygons $\Delta^{(1)}$ having at most $32$ lattice points or having lattice width at most $6$. 
Thus the claim follows from Lemma \ref{lemma2conj} and Theorem~\ref{cliffindexthm}.
\end{proof}

\begin{remark} \label{countergreenremark}
The inequality $|\partial\Delta^{(1)}\cap\ZZ^2| \geq \lw(\Delta^{(1)})+2$ is satisfied for the vast majority of polygons. In fact it is not so easy to find polygons for which the inequality is \emph{not} satisfied, where of course the condition of being interior is crucial: if we omit this assumption, it is trivial to find counterexamples (e.g., the primitive lattice triangles that we encountered at the end of Section~\ref{section_constancyresults} 
can have arbitrarily large lattice width). The smallest interior counterexample that 
we encountered is $\Delta^{(1)}$ where $\Delta = \conv \{ (4,0), (0,10), (10,4) \}$. This concerns a $9$-gon without extra points on the boundary, satisfying
$g=|\Delta^{(1)}\cap \mathbb{Z}^2|=36$ and $\lw(\Delta^{(1)})=8$, see Figure~\ref{triangle}.
\begin{figure}[h] 
\centering
\begin{tikzpicture}[scale=0.4]
  \draw[lightgray] (0,0) grid (10,10);
  \draw[thick] (4,0) -- (10,4) -- (0,10) -- (4,0);
  \draw[dashed] (4,1) -- (5,1) -- (8,3) -- (9,4) -- (8,5) -- (3,8) -- (1,9) -- (1,8) -- (3,3) -- (4,1);
  \draw[fill] (4,1) circle [radius=0.15];
  \draw[fill] (5,1) circle [radius=0.15];
  \draw[fill] (8,3) circle [radius=0.15];
  \draw[fill] (9,4) circle [radius=0.15];
  \draw[fill] (8,5) circle [radius=0.15];
  \draw[fill] (3,8) circle [radius=0.15];  
  \draw[fill] (1,9) circle [radius=0.15];
  \draw[fill] (1,8) circle [radius=0.15];
  \draw[fill] (3,3) circle [radius=0.15];
  \node at (6,8) {$\Delta$};
\end{tikzpicture}
\caption{Counterexample to the inequality $|\partial \Delta^{(1)}\cap \mathbb{Z}^2|\geq \lw(\Delta^{(1)})+2$}
\label{triangle}
\end{figure}
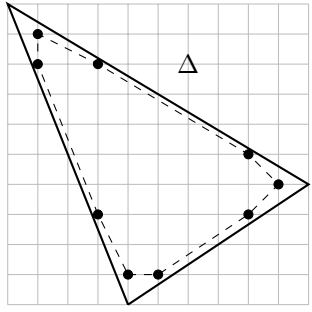
If we want to check Green's conjecture for this specific polygon $\Delta$, we need to show that $a_{27}=0$ (since $g-(\gamma-1)=27$). Using the algorithm from~\cite{previouspaper} we checked that $b_{27}=0$, therefore Conjecture~\ref{conj_CCDL} holds in this case. On the other hand $c_{27}\neq 0$ by Hering and Schenck's result, so we cannot use \eqref{sixterms} to conclude that $a_{27} = 0$. In fact if the sum formula $a_{27} = b_{27} + c_{27}$ from the statement of Theorem~\ref{maintheorem} 
would be true in this case (which we do not believe is the case), then from $c_{27} \neq 0$ it would follow that $a_{27} \neq 0$ and hence that Green's conjecture is false!
\end{remark}

\subsection*{New cases of Conjecture~\ref{conj_CCDL}
from known cases of Green's conjecture} 

%As a second application, we use known cases of Green's conjecture to deduce new cases of Conjecture~\ref{conj_CCDL}.
Here the main input is due to
Lelli-Chiesa, who proved Green's conjecture for curves on smooth rational surfaces, modulo certain assumptions, the most restrictive one being the existence of an anticanonical pencil. 
Let us state her result more precisely, adapting the notation to our specific case of curves of the form $C_f \subseteq X_\Delta$. Because the ambient surface needs to be smooth, as in Section~\ref{sectionsixterms} we let $X\to X_\Delta$ be the minimal toric resolution of singularities and write $C'$ for the pull-back of $C_f$. Again we
let $D_f = C' - \Div(f)$ and consider the canonical divisor $K=-\sum_v D_v$, where $v$ ranges over the set $U(\Sigma)$ of primitive generators of the rays of the fan $\Sigma$ of $X$.

\begin{theorem}[see \cite{lellichiesa}] \label{lellichiesatheorem} Assume that the following conditions are satisfied:
\begin{itemize}
\item $L=D_f+K$ is big and nef,
\item $h^0(X,-K)\geq 2$,
\item if $h^0(X,-K)=2$, then the Clifford index of a general curve $C \in |D_f|$ is not computed by restricting the anticanonical divisor to $C$.
\end{itemize}
Then Green's conjecture is true for $C_f$.
\end{theorem}

Note that the second condition can be rephrased as $|P_{-K}\cap\mathbb{Z}^2|\geq 2$.
The first condition is automatically satisfied for toric surfaces: $L$ is nef because of Lemma \ref{fujitalemma} and big because $\Delta^{(1)}$ is two-dimensional.
The next two lemma's show that also the third condition is void in our case. 

\begin{lemma} \label{lemma_lellichiesacondition}
Let $X$ be a toric surface with $h^0(X,-K)=2$. If $\Delta=P_D$ is the polygon of a torus-invariant nef Cartier divisor $D$ on $X$, then 
$\textup{lw}(\Delta)<|\partial\Delta\cap\ZZ^2|$.
\end{lemma}
\begin{proof}
The fact that $D$ is nef and Cartier ensures that $\Delta$ is a lattice polygon. Now
there is a non-zero lattice point $m_0\in P_{-K}$ and by using a unimodular transformation if needed, we can assume that $m_0=(1,0)$. 
Let $y_1$ (resp.\ $y_2$) be the minimum (resp.\ maximum) of the second coordinates of the points of $\Delta$.
For all 
$v\in U(\Sigma)$, we have that $\langle m_0,v\rangle\geq-1$, hence the first coordinate of each $v\in U(\Sigma)$ is at least $-1$. Consider an edge $e$ at the right hand side of $\Delta$, i.e., an edge
whose inner normal vector has a strictly negative first coordinate. Then the corresponding $v\in U(\Sigma)$ must have first coordinate equal to $-1$. Hence, if $e$ intersects a horizontal line at integral height, then this point of intersection is 
a lattice point. As a consequence $|\partial\Delta\cap\ZZ^2|> y_2-y_1 \geq \lw(\Delta)$. 
\end{proof}

\begin{lemma} \label{lem_ciK}
Let $X$ be a smooth toric surface with $h^0(X,-K)=2$. Let $D$ be a torus-invariant nef divisor on $X$ such that $D+K$ is big and nef. Then for a general curve $C\in |D|$ the Clifford index is not computed by $-K|_C$.
\end{lemma}

\begin{proof}
Note that all divisors are Cartier because of the smoothness assumption. Denote the (lattice) polygon $P_D$ corresponding to $D$ by $\Delta$. 
The short exact sequence $0 \to \mathcal{O}_X(-D-K) \to \mathcal{O}_X(-K) \to \mathcal{O}_C(-K|_C)\to 0$
%$\omega_S^{\vee}\otimes L^\vee\rightarrow\omega_S^{\vee}\rightarrow\omega_S^{\vee}|_C$ 
yields the long exact sequence 
$$0 \to H^0(X,-D-K)\to H^0(X,-K)\to H^0(C,-K|_C)\to H^1(X,-D-K) \to \ldots$$
Since $D+K$ is big and nef, the polygon $P_{D+K} = \Delta^{(1)}$ is two-dimensional and we have that $h^0(X,-D-K)=h^1(X,-D-K)=0$ by Batyrev-Borisov vanishing. 
It follows that $h^0(C,-K|_C)=h^0(X,-K)=2$. Hence, the divisor $-K|_C$ gives rise to a linear system on $C$ of rank $r=h^0(C,-K|_C)-1=1$ and degree 
$\sum_{v\in U(\Sigma)} \deg(D_v|_C)=|\partial \Delta\cap\mathbb{Z}^2|$. 
Now if the Clifford index of $C$ would be computed by $-K|_C$, then we would have $\ci(C)=|\partial \Delta\cap\mathbb{Z}^2|-2$. On the other hand, by Theorem~\ref{cliffindexthm} and Lemma~\ref{lemma_lellichiesacondition}, we have that $\ci(C)\leq\lw(\Delta^{(1)})\leq\lw(\Delta)-2<|\partial\Delta\cap\ZZ^2|-2$, a contradiction.
\end{proof}

We can now conclude with a proof of Theorem~\ref{thm_newcasestoricconjecture}, which we reformulate as the following corollary:

\begin{corollary}
If 
\[ |P_{-K_{\Delta^{(1)}}} \cap\mathbb{Z}^2|\geq 2, \] then Conjecture~\ref{conj_CCDL} correctly predicts the length of the linear strand
of the graded Betti table of $X_{\Delta^{(1)}}$.
\end{corollary}
\begin{proof}
  Because the statement only involves $\Delta^{(1)}$, we can assume that $\Delta$ is maximal.
  Using a unimodular transformation if needed, we can also assume that $(1,0)$ is contained in the polygon associated with $-K_{\Delta^{(1)}}$, which implies,
  as in the proof of Lemma~\ref{lemma_lellichiesacondition}, that all inner normal vectors to $\Delta^{(1)}$ having a strictly negative first coordinate must be of the form $(-1,b)$ for some $b \in \mathbb{Z}$.
  But then the same must be true for $\Delta = \Delta^\text{max}$, which is obtained from $\Delta^{(1)}$ by moving out the edges over an integral distance $1$. We claim that
  the minimal toric resolution of singularities $X \rightarrow X_\Delta$ is obtained by inserting rays whose primitive generators are of the form $(a,b)$ with $a \geq -1$.
  Indeed,
  \begin{itemize}
    \item minimally subdividing a cone spanned by $(a_1,b_1)$ and $(a_2,b_2)$ for integers $a_1,a_2,b_1,b_2$ with $a_1,a_2 \geq 0$ clearly introduces such rays only,
    \item minimally subdividing a cone spanned by $(-1,b_1)$ and $(a_2,b_2)$ for integers $a_2,b_1,b_2$  with $a_2 \geq 0$ introduces the ray spanned by $(0,-1)$ and rays whose primitive generators have a positive first coordinate,
    \item minimally subdividing a cone spanned by $(-1,b_1)$ and $(-1,b_2)$ for integers $b_1<b_2$ introduces the rays spanned by all integral vectors of the form $(-1, b)$ with $b_1 < b < b_2$.
  \end{itemize}
  In other words $(1,0) \in P_{-K}$, and therefore we can apply Lelli-Chiesa's theorem to conclude that Green's conjecture is true for
  any smooth hyperplane section $C_f \subseteq X_\Delta$.
  The conclusion now follows from Lemma \ref{lemma2conj}.
\end{proof}

 As a special case we find that Conjecture~\ref{conj_CCDL} is true if $X_{\Delta^{(1)}}$ is Gorenstein weak Fano.

\section*{Acknowledgements}
This research was partially supported by the research project G093913N of the Research Foundation Flanders (FWO), and by
the European Commission through the European Research
Council under the FP7/2007-2013 programme with ERC Grant Agreement 615722 MOTMELSUM.
%The first author is affiliated on a voluntary basis with the research group imec-COSIC at KU Leuven and with the Department of Mathematics: Algebra and Geometry at Ghent University. 
The fourth author was supported by a PhD fellowship of the Research Foundation Flanders (FWO) during the preparation of this manuscript.
All authors thank Imre B\'ar\'any for answering a question of ours, which eventually led to the counterexample in Remark~\ref{countergreenremark}, and Margherita Lelli-Chiesa
for a helpful discussion. They thank the anonymous referees of the current submission and an anonymous referee of a previous submission for their careful reading and useful remarks, which were of significant help in improving the exposition.

\bibliographystyle{amsplain} 
\bibliography{can_syzygies}  
%\nocite{label}
\small 

\vspace{0.3cm}

\noindent \textsc{Department of Mathematics, KU Leuven, Celestijnenlaan 200B, 3001 Leuven (Heverlee), Belgium}

\noindent and

\noindent \textsc{Department of Electrical Engineering and imec-COSIC, KU Leuven, Kasteelpark Arenberg 10/2452, 3000 Leuven (Heverlee), Belgium}

\noindent and

\noindent \textsc{Department of Mathematics: Algebra and Geometry, Ghent University, Krijgs\-laan 281/S23, 9000 Gent, Belgium}

\noindent \texttt{wouter.castryck@esat.kuleuven.be}\\

\noindent \textsc{Department of Mathematics, KU Leuven, Celestijnenlaan 200B, 3001 Leuven (Heverlee), Belgium}

\noindent \texttt{f.cools@kuleuven.be}, \texttt{alexander.lemmens@kuleuven.be}\\

\noindent \textsc{Department of Mathematics: Algebra and Geometry, Ghent University, Krijgs\-laan 281/S23, 9000 Gent, Belgium}

\noindent \texttt{j.demeyer@ugent.be}
\end{document}